\documentclass[11pt]{article}
\usepackage{layout}
\usepackage{amssymb}
\usepackage{epsfig}
\usepackage{graphicx}
\usepackage{color}
\usepackage{latexsym}

\usepackage{amssymb}
\usepackage{algorithmicx}
\usepackage[ruled]{algorithm}
\usepackage{algpseudocode}
\usepackage{algpascal}
\usepackage{algc}

\alglanguage{pseudocode}
\usepackage{graphicx}
\usepackage{amsmath,epsfig,graphics}

\newenvironment{proof}[1][Proof]{\textbf{#1.} }{\ \rule{0.5em}{0.5em}}
\newcommand{\R}{I\!\!R}

\newtheorem{theorem}{Theorem}[section]

\newtheorem{corollary}[theorem]{Corollary}

\newtheorem{proposition}[theorem]{Proposition}
\newtheorem{remark}[theorem]{Remark}

\newcommand{\Frac}[2] {\frac{\textstyle #1} {\textstyle #2}}

 \newcommand{\Max}  {\mathop{\rm Max}}

\begin{document}

\title{Stabilized Times Schemes for  High Accurate Finite Differences Solutions of Nonlinear Parabolic Equations}
\maketitle
\centerline{\scshape Matthieu Brachet$^{1}$ and Jean-Paul Chehab$^2$ 
}
\medskip
\centerline{$^{1}${\footnotesize Institut Elie Cartan de Lorraine, Universit\'e de Lorraine,  }} 
 \centerline{{\footnotesize Site de Metz, B\^at. A Ile du Saucy, F-57045 Metz Cedex 1 }}
{
\centerline{$^{2}${\footnotesize Laboratoire Amienois de Math\'ematiques Fondamentales et Appliqu\'ees (LAMFA), {\small UMR} 7352}}
  \centerline{{\footnotesize Universit\'e de Picardie Jules Verne, 33 rue Saint Leu, 80039 Amiens France} }

\begin{abstract}
The Residual Smooting Scheme (RSS) have been introduced in \cite{AverbuchCohenIsraeli} as a backward Euler's method with a simplified implicit part for the solution of parabolic problems. RSS have stability properties comparable to those of  semi-implicit schemes while giving possibilities for reducing  the computational cost. A similar approach was introduced independently in \cite{BCostaPHD,CDGT} but from the Fourier point of view. We present here
a unified  framework for these schemes and propose practical  implementations and extensions of the RSS schemes for the long time simulation of nonlinear parabolic problems when discretized by using high order finite differences compact schemes. Stability results are presented in the linear and the nonlinear case. Numerical simulations of 2D incompressible Navier-Stokes equations are given for illustrating the robustness of the method.
\end{abstract}
{\small 
}\\
\section{Introduction}
It is a common fact in numerical analysis that the choice of a time
marching scheme must balance stability, accuracy and reasonable computational
cost. Typically, when considering e.g. the numerical solution of a space-discretized
parabolic problems, such as
\begin{equation}
\begin{array}{l}
\Frac{du}{dt} +Au=f,\\
u(0)=u_0,
\end{array}
\end{equation}
where $A$ denotes the stiffness matrix, it is well known that  the implicit schemes are stable but need an
additional problem to be solved at each step while the explicit
schemes are very cheap but suffer of a hard time step limitation
making them bad suited for capturing the long time behavior of the
solutions. An ideal scheme should combine stability and low
computational cost (explicity) for a comparable accuracy.\\

Two independent attempts have been made successfully in that
direction :

First, A. Cohen and {\it al}  proposed in 
\cite{AverbuchCohenIsraeli} the following stabilization to the
forward Euler scheme (Residual Smoothing Scheme or RSS) for the discretized parabolic equations associated to homogeneous Dirichlet (or Neumann) Boundary conditions:
\begin{equation}
\begin{array}{lcl}
\Frac{u^{(k+1)}-u^{(k)}}{\Delta t}+&\tau
\underbrace{B(u^{(k+1)}-u^{(k)})}&+Au^{(k)}=f,\\
 & \mbox{Stabilization term}& \\
 \end{array}
\end{equation}
where $\tau$ is a positive real number to be chosen and $B$ a preconditioner of the stiffness matrix $A$.
Originally introduced in the context of wavelet discretizations, the
matrix $B$ can be taken as the diagonal part of $A$ and is then an inconditional preconditioner: the new scheme
is no more expensive than the classical forward Euler's while the
stability is increased. However, RSS is only first order accurate in
time and, in order to increase the accuracy, it was proposed in
\cite{AverbuchCohenIsraeli} to apply a Richardson extrapolation; typically a second order of accuracy was
obtained as shown by numerical evidences. A rough analysis of the
stabilized and extrapolated scheme was made by Ribot and Schatzman
\cite{MRibotMSchatzman,RibotSchatzman2}, they derived stability and error estimates in energy norms .

Independently, Costa, Dettori, Gottlieb and Temam have introduced in
\cite{BCostaPHD,CDGT} a similar approach but starting from a
Fourier-analysis point of view, in the context of multiresolution methods of
nonlinear Galerkin type for spectral discretizations (Fourier,
Chebyshev), see \cite{TemamIDS}. They proposed to stabilize the forward Euler scheme for
the heat equation
$$
u_t -u_{xx}=f,
$$
by adding a stabilization term of the form $\beta\Frac{du}{dt}$. The
new scheme writes in Fourier basis as
\begin{equation}
\begin{array}{lcl}
\Frac{{\hat u}^{(k+1)}_m-{\hat u}^{(k)}_m}{\Delta t}& +\underbrace{\Frac{\beta_j}{\Delta t}({\hat u}^{(k+1)}_m-{\hat u}^{(k)}_m)}
&=-m^2{\hat u}^{(k)}_m + {\hat f}_m, \ m=N_{j-1}, \cdots, N_{j},\\
 & \mbox{Stabilization term}& \\
 \end{array}
\end{equation}
where we have decomposed the frequency range $[1,N]$ into $\displaystyle{\cup_{j=1}^{d}[N_{j-1},N_{j}]}$. In other words
$$
\Frac{{\hat u}^{(k+1)}-{\hat u}^{(k)}}{\Delta t} -B({\hat u}^{(k+1)}-{\hat u}^{(k)})
=-D{\hat u}^{(k)} + {\hat f}, \
$$
with $D=diag(1,4,9, \cdots N^2)$ and 
$$
B=
\left(
\begin{array}{llll}
B_1 & 0 & & 0\\
0 & \ddots &  & \\
 & & \ddots &   \\
 0 & &  & B_d\\
\end{array}
\right),
$$
with $B_j=\beta_j Id_j$, $Id_j$ being the identity matrix of size $(N_{j}-N_{j-1}) \times (N_{j}- N_{j-1})$ (we have set $N_0=0$ and $N_d=N$);
$B$ is then a preconditioner of $D$ in the Fourier space. In the linear case, these two approaches coincide. Of course, same framework can be derived when considering orthogonal polynomials.
Costa and Chehab \cite{ChehabCosta1,ChehabCosta2} have
extended this scheme to hierarchical discretizations in finite
differences.\\

The main advantage of the RSS approach is that a simplified (yet costless) solver is used for the implicit part of the time marching scheme while displaying comparable stability properties to Backward's Euler Scheme. One situation of particular interest, on which we focus in the present work, occurs when handling high order discretizations of the stiffness matrix $A$, e.g. with  finite differences compact schemes. In that case,  $A$ is full, this is due to the implicit part of the scheme. Hence, matrix-vector product are costlty and must be reduced as far as possible.  A lower level of space dicretization (say second order) generates a sparse stiffness matrix $B$ which is a natural efficient preconditioner of $A$.  Then, the RSS scheme can be implemented efficiently taking advantage of the existing  (sometimes fast) solvers of the system of the form 
$$
(Id +\Delta t B) u=f,
$$
such as sparse factorizations and FFT.
\\
The RSS approach can be proposed to solve fully discretized time dependent PDEs with high accurate spatial discretization, with compact scheme, while using the computational facilities of the sparse numerical solvers (Fast solvers, limited memory).\\

In this article, we propose a unified approach to RRS-like schemes that rely \cite{AverbuchCohenIsraeli} and \cite{CDGT}. We derive stability results in the linear and the nonlinear case;  we also present practical and efficient adaptations for the high accurate finite differences solutions of nonlinear parabolic equations.

The paper is organized as follows:
in Section 2,  we derive a general approach to RSS schemes and we give stability results in the linear and the nonlinear case,  the accuracy in time of the new scheme is also discuted. After that, in section 3, we describe the compact scheme discretization and the preconditioning that will be used. Section 4 is devoted to numerical illustrations, we compare RSS approach to the classical one with emphasis on the stability, the accuracy (particularly the dynamics to steady state) and for that purspose, we solve high accurate finite difference steady state of 2D incompressible Navier-Stokes equations (Lid driven cavity) and recover the results of the litterature.
\section{Derivation of the stabilized schemes - properties}
\subsection{Formal Derivation of the stabilized schemes}
Let us consider the finite dimensional differential system
\begin{eqnarray}
\Frac{du}{dt}+F(u)=0, t>0,\\
u(0)=u_0,
\end{eqnarray}
here $F : \R^N\rightarrow \R^N$ is a regular map. The backward Euler
scheme applied to the above system generates the iterations
$$
u^{(k+1)}-u^{(k)}+\Delta tF(u^{(k+1)})=0,
$$
and a (possibly) nonlinear problem must be solved at each step. Making
the approximation
$$
F(u^{(k+1)})\simeq F(u^{(k)}) +F'(u^{(k)})(u^{(k+1)}-u^{(k)}),
$$
where $F'(u^{(k)})$ denotes the differential of $F$ at $u^{(k)}$, we obtain the scheme
$$
\Frac{u^{(k+1)}-u^{(k)}}{\Delta t} + F'(u^{(k)})(u^{(k+1)}-u^{(k)})
+F(u^{(k)})=0,
$$
so
$$
u^{(k+1)}=u^{(k)}-\Delta t(Id +\Delta tF'(u^{(k)}))^{-1}F(u^{(k)}).
$$
Setting $\Phi(v)=v-u^{(k)}+\Delta tF(v)$, we see that $u^{(k+1)}$
is nothing else but the first iteration of the Newton-Raphson scheme
applied to $\Phi(v)$ when starting from the initial guess $u^{(k)}$.

Now, if we replace  $F'(u^{(k)})$ by a preconditioner $\tau B_k$, we
find
\begin{equation}
\begin{array}{lll}
\Frac{u^{(k+1)}-u^{(k)}}{\Delta t}+&\tau
\underbrace{B_k(u^{(k+1)}-u^{(k)})}&+F(u^{(k)})=0,\\
&\mbox{Global stabilization}& \\
\end{array}
\label{NLRSS1}
\end{equation}
and $u^{(k+1)}$ is thus the first iteration of a quasi Newton Method
applied to $\Phi(v)$ when starting from the initial guess $u^{(k)}$.\\

The efficiency of this stabilized scheme is closely related to the
cost of the computation of the preconditioner of the jacobian matrix which
changes at each iteration: technique of existing updating 
factorizations as those presented in \cite{CalgaroChehabSaad} and \cite {Bellavia} could
be adapted.\\

In the present work, we will not discuss on the analysis of the nonlinear version of the scheme, say (\ref{NLRSS}), but we will present on Section 4 numerical results obtained with this scheme. We will rather consider the semi linear approach: if $F(u)$ can be expressed as $F(u)=Au+f(u)$, we define teh scheme
\begin{equation}
\begin{array}{lll}
\Frac{u^{(k+1)}-u^{(k)}}{\Delta t}+&\tau
\underbrace{B(u^{(k+1)}-u^{(k)})}&+F(u^{(k)})=0,\\
&\mbox{Stabilization of the linear part}&\\
\end{array}
\label{NLRSS}
\end{equation}
where $B$ is a preconditioner of $A$\\

It is important to point out that (\ref{NLRSS}) is consistant with the computation of steady states and can then be 
applied as a pseudo-time numerical solver, as illustrated in Section 4 with impressible NSE.\\

Of course this stabilization approach applies to the linear case $(f(u)=0)$. Particularly, 
RSS is a simplified $\theta$-scheme in which the matrix $A$ is replaced by a
preconditioner. Indeed, the $\theta$-scheme write, after
simplifications as
$$
u^{(k+1)}=u^{(k+1)}-\Delta t(Id +\theta\Delta t A)^{-1}(Au^{(k)}-f)
$$
and, substituting $A$ by $B$ in the implicit part, we recover the RSS
$$
u^{(k+1)}=u^{(k+1)}-\Delta t(Id +\theta\Delta t B)^{-1}(Au^{(k)}-f)
$$
with $\theta=\tau$.\\

In practice, the use of a preconditioner $B$ of $A$ leads to propose
$K=Id +\tau \Delta t B$ as preconditioner of $M=Id +\Delta t A$, where $Id$ is the identity matrix. This can be realized
in many ways, e.g., by computing $K$ as an incomplete factorization of $M$; in some cases it can be done by solving the linear systems involving $K$ with fast solvers (FFT or so), see section 3. The RSS approach applies also to linear problems with a matrix $A(t)$ which depends on time $t$: 
\begin{equation}
\Frac{\partial u}{\partial t} +A(t)u=f,
\end{equation}
that we discretize as
\begin{equation}
\begin{array}{cll}
\underbrace{(Id +\tau \Delta t B_k)}&(u^{(k+1)}-u^{(k)})&=F-A(k \Delta t)u^{(k).}\\
\mbox{$K_k$} & & 
\end{array}
\label{TimeDepRSS1}
\end{equation}
The matrix $K_k$ can be computed as an incomplete LU factorization of $M=Id +\Delta t A(k \Delta t)$ and, if $A(t)$ does vary slightly with $t$, incremental factorization updates from $K_{k-1}$ can be done following the techniques proposed in \cite{CalgaroChehabSaad}.
Notice also that scheme (\ref{TimeDepRSS}) can be obtained by applying RSS to linearized equation, as
\begin{equation}
\begin{array}{cll}
\underbrace{(Id +\tau \Delta t B_k)}&(u^{(k+1)}-u^{(k)})&=-\Delta t F(u^{(k)}),\\
\mbox{$K_k$} & & 
\end{array}
\label{TimeDepRSS}
\end{equation}
where $B_k$ is here such that $F(u^k)=B_k u^k$.
\subsection{Properties of the schemes}
\subsubsection{The linear case}
Let $A$ and $B$ be both $N\times N$  real symmetric positive definite matrices; the symmetry of $A$ is considered for the sake of simplicity however  the following approach remains valuable in the nonsymmetric case, see section 3 and Theorem \ref{theo_gen_stab}, . We assume that there exist two strictly positive real numbers $\alpha$ and $\beta$ such that
$$
{\cal H}\hskip 2.cm  \alpha <Bu,u>\le <Au,u> \le \beta <Bu,u>, \ \forall u \in {\mathbb R}^N.
$$
It is important to note that $\alpha$ and $\beta$ can depend on the dimension $N$, if not the matrix $B$ is said to be an inconditional preconditioner of $A$.
We will use the following notations: $<.,.>$ is the euclidian scalar product in ${\mathbb R}^N$ and 
$\parallel .\parallel$, the associated norm.
We will note $\lambda_{min}$ (resp. $\lambda_{max}$) the lowest (resp. the largest) eigenvalue of $A$.\\

We now consider the RSS scheme applied to the discretized heat equation
$$
 \Frac{u^{(k+1)}-u^{(k)}}{\Delta t} +\tau B (u^{(k+1)}-u^{(k)}) =-Au^{(k)}.
 $$
We first prove a simple stability result:
\begin{proposition}
Under hypothesis ${\cal H}$, we have the following stability conditions:
\begin{itemize}
\item If $\tau\ge \Frac{\beta}{2}$, the scheme is unconditionally stable (i.e. stable $\forall \ \Delta t >0$)
\item If $\tau < \Frac{\beta}{2}$, then the scheme is stable for
$
0<\Delta t < \Frac{2}{\left(1-\Frac{2\tau}{\beta}\right)\rho(A)}.
$
\end{itemize} 
\label{RSS_Stab_lin}
\end{proposition}
\begin{proof}
Taking the usual scalar product of each terms with $u^{(k+1)}-u^{(k)}$, we find
$$
\Frac{1}{\Delta t} \parallel u^{(k+1)}-u^{(k)} \parallel^2
+\tau <B(u^{(k+1)}-u^{(k)}),u^{(k+1)}-u^{(k)}> = -<Au^{(k)},u^{(k+1)}-u^{(k)}>.
$$
Using the parallelogram identity, 
$$
<Au^{(k)},u^{(k+1)}-u^{(k)}>=\Frac{1}{2}\left( <Au^{(k+1)},u^{(k+1)}>-<A(u^{(k)},u^{(k)}+<A(u^{(k+1)}-u^{(k)}),u^{(k+1)}-u^{(k)})>\right),
$$
we infer
$$
\begin{array}{ll}
\Frac{1}{\Delta t} \parallel u^{(k+1)}-u^{(k)} \parallel^2
+\tau <B(u^{(k+1)}-u^{(k)}),u^{(k+1)}-u^{(k)}> & \\
&\\
-\Frac{1}{2}\left(<Au^{(k)},u^{(k)}> -<Au^{(k+1)},u^{(k+1)}>+
<A(u^{(k+1)}-u^{(k)}),u^{(k+1)}-u^{(k)}>\right)&=0.\\
\end{array}
$$
Hence the stability condition $<Au^{(k)},u^{(k)}> \ >  \ <Au^{(k+1)},u^{(k+1)}>$ holds when
$$
\Frac{1}{\Delta t} \parallel u^{(k+1)}-u^{(k)} \parallel^2
+\tau <B(u^{(k+1)}-u^{(k)}),u^{(k+1)}-u^{(k)}> 
-\Frac{1}{2}
<A(u^{(k+1)}-u^{(k)}),u^{(k+1)}-u^{(k)}>>0.
$$
We have, using ${\cal H}$
$$
\begin{array}{c}
\tau <B(u^{(k+1)}-u^{(k)}),u^{(k+1)}-u^{(k)}> 
-\Frac{1}{2}
<A(u^{(k+1)}-u^{(k)}),u^{(k+1)}-u^{(k)}>\\
\ge \\
 \left(\Frac{\tau}{\beta}-\Frac{1}{2}\right) <A(u^{(k+1)}-u^{(k)}),u^{(k+1)}-u^{(k)}> (\ge 0).
 \end{array}
$$
A sufficient stability condition is then
$$
\Frac{1}{\Delta t} \parallel u^{(k+1)}-u^{(k)} \parallel^2
+ \left(\Frac{\tau}{\beta}-\Frac{1}{2}\right) <A(u^{(k+1)}-u^{(k)}),u^{(k+1)}-u^{(k)}> >0.
$$
This is satisfied once $\Frac{1}{\Delta t}+\lambda_{min}(\Frac{\tau}{\beta}-\Frac{1}{2})>0$.
Therefore, if $\Frac{\tau}{\beta}-\Frac{1}{2} \ge 0$ the previous inequality holds for all $\Delta t >0$, this means
the stability $\forall \Delta t >0$.


Now if $\tau < \Frac{\beta}{2}$, then, since $<A(u^{(k+1)}-u^{(k)}),u^{(k+1)}-u^{(k)}> \le \rho(A)  \parallel u^{(k+1)}-u^{(k)} \parallel^2$, a sufficient condition of stability is
$$
\Frac{1}{\Delta t}
+\left(\Frac{\tau}{\beta}-\Frac{1}{2}\right)\rho(A) >0,
$$
from which we deduce
$$
0<\Delta t < \Frac{2}{\left(1-2\Frac{\tau}{\beta}\right)\rho(A) },
$$
as a sufficient stability condition.
\end{proof}
\\
We point out that if $B=A$ (then $\alpha=\beta=1$) and $\tau=\theta \in [0,1]$, the stability condition coincide with the one of the $\theta$-scheme.\\
The stability is easily obtained when taking $\tau$ large enough. However, a too large value of $\tau$ deteriorates the consistency of the scheme and, as a particular effect, the convergence to the steady state is longer in time.
In fact both the value of $\tau$ and the preconditioning quality of $B$ act on the accuracy of the RRS scheme which remains always first order accurate in time as illustrated in Figure (\ref{tau_effect}).
\\
\begin{figure}[!h]
\begin{center}
\includegraphics[width=7cm]{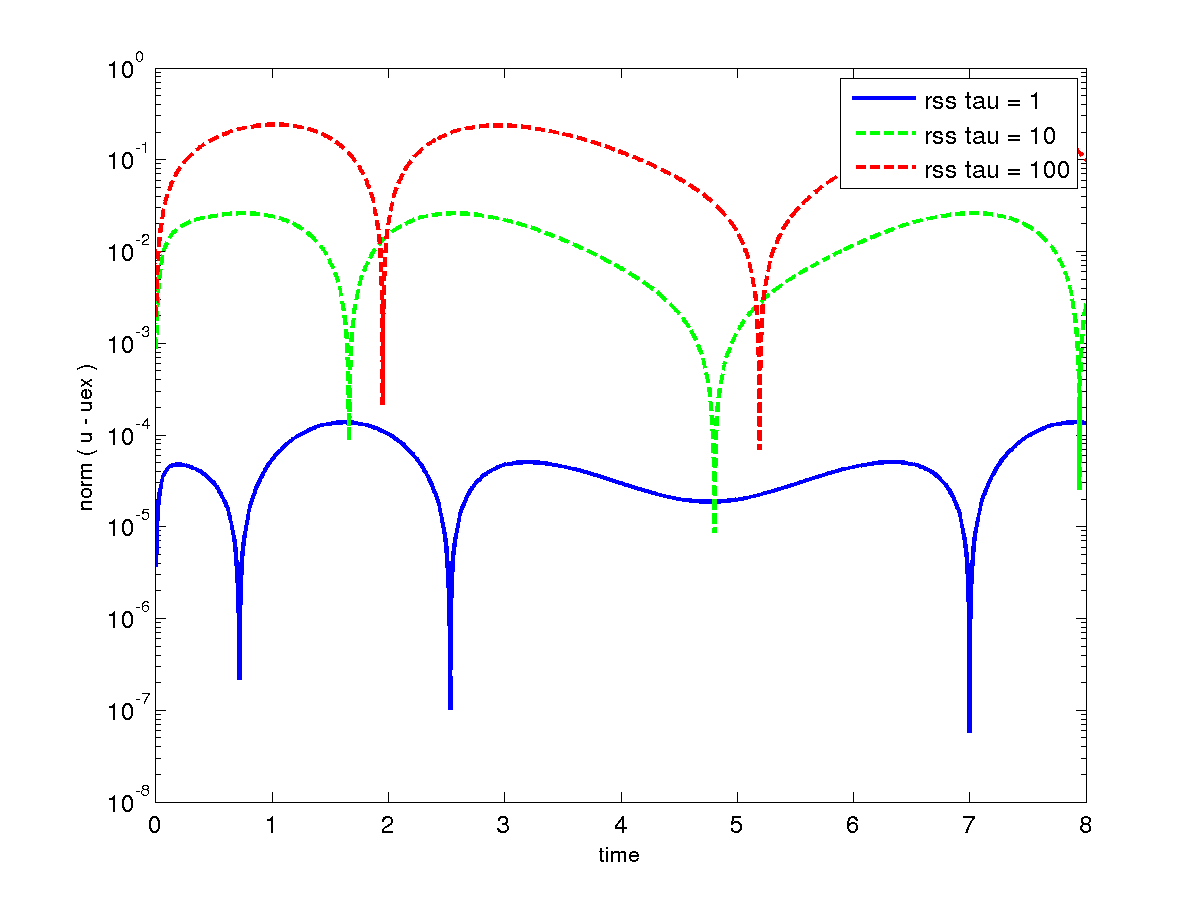}
\includegraphics[width=7cm]{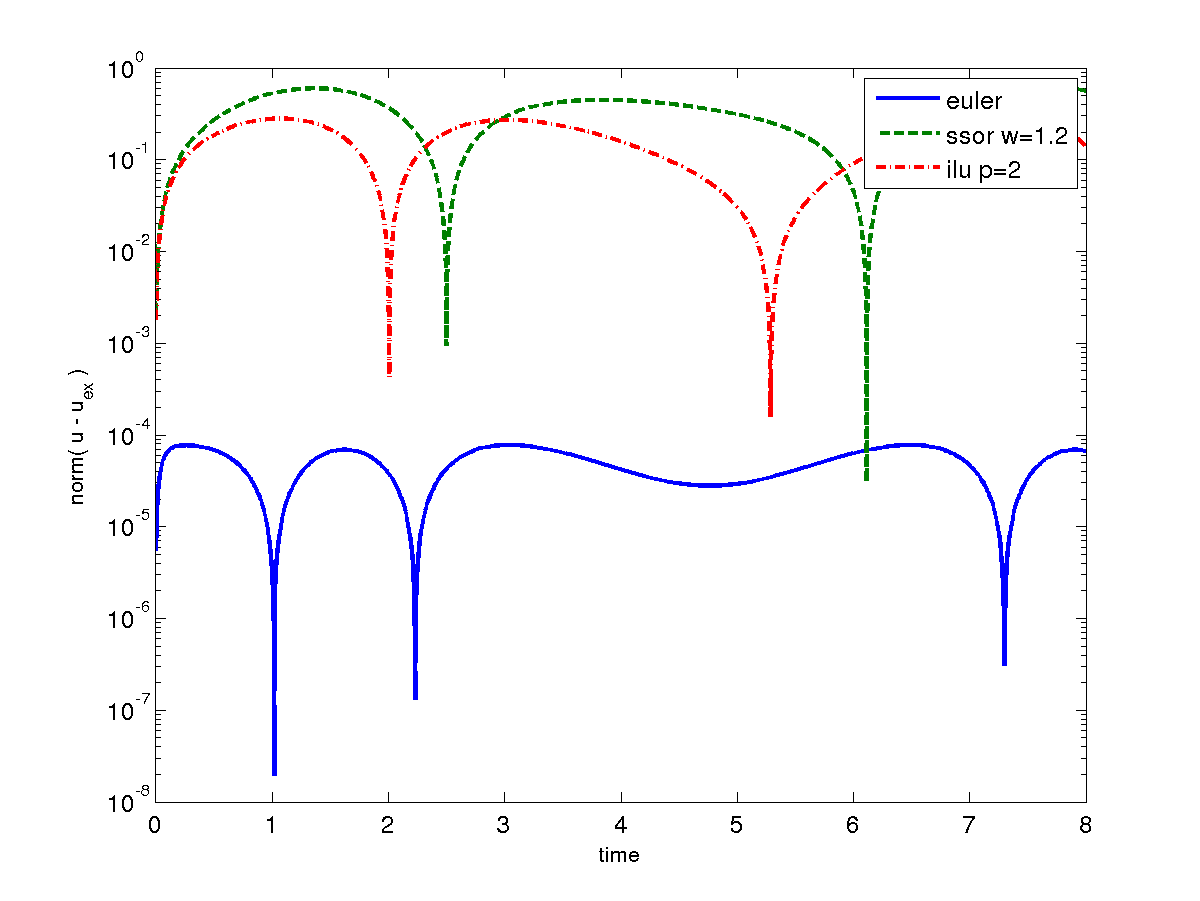}\\
\caption{(left) Influence of the parameter $\tau$ on the accuracy of RSS. Error vs time for different values of $\tau$  - $N=127$, $\Delta t=0.004$, (right) Error vs time with ILU(2) and SSOR preconditioners, $N=63$,$\tau1$, $\Delta t=0.004$}
\label{tau_effect}
\end{center}
\end{figure}
\\

A first natural question is the choice of the best value $\tau_{opt}$ of $\tau$, for a fixed preconditioner $B$; $\tau_{opt}$ can be simply computed such as minimizing $\parallel \tau B -A\parallel_F$. We easily show that
$$
\tau_{opt}=\Frac{trace(B^TA)}{\parallel B\parallel^2_F}.
$$
We remark that when $B=Id$, then $\tau_{opt}=\Frac{trace(A)}{n}$, the mean value of the eigenvalues of $A$.

A second natural question deals with the gain of stability brought by the RSS scheme as respect to an explicit method, namely forward Euler's for which the time step must be taken strictly lower than $\Delta t_c=\Frac{2}{\rho(A)}$.
In other words, for a given preconditioner matrix $B$ and for a given number $\tau$, we look to $\kappa>1$ such that
RSS is stable with time step $\Delta t =\kappa \Delta t_c$.
We deduce directly from the previous computations.
\begin{proposition}
\begin{itemize}
\item If $\tau\ge \Frac{\beta}{2}$, RSS is infinitely more stable than backward Euler's.
\item If $\tau < \Frac{\beta}{2}$, then RSS is at least $\kappa=\Frac{1}{1-\Frac{2\tau}{\beta}}$ times more stable than Euler's.
\end{itemize} 
\end{proposition}
\begin{proof}
We deduce from the proposition \ref{RSS_Stab_lin} that 
$
\Frac{2}{\left(1-\Frac{2\tau}{\beta}\right)\rho(A)}=\kappa \Frac{2}{\rho(A)},
$
then
$
\kappa=\Frac{1}{\left(1-\Frac{2\tau}{\beta}\right)}.
$
\end{proof}
\\

We now propose to quantify the consistency error with $\tau$ by comparison with the Backward Euler Scheme (which is unconditionally stable). Particularly we analyse the behavior of the difference of the sequences generated by the two schemes: the stabilization has as an effect to slow down the convergence in time to the steady state.

\begin{proposition}
We consider the two sequences
$$
 \Frac{u^{(k+1)}-u^{(k)}}{\Delta t} +\tau B (u^{(k+1)}-u^{(k)}) =f-Au^{(k)},
 $$
and
$$
 \Frac{v^{(k+1)}-v^{(k)}}{\Delta t} +A v^{(k+1)} =f,
 $$
 with $u^{(0)}=v^{(0)}$. We let $M=Id-\Delta t(Id+\tau \Delta t B)^{-1}A$ and we assume that $\parallel M\parallel < 1$,
 then,  there exists $\gamma \in [0,1[$ such that
 $$
\parallel u^{(k)}- v^{(k)}\parallel \le  \Delta t^2 \parallel  \tau B-A\parallel \Frac{1}{1-\gamma}\parallel  f-Av^{(0)} \parallel  , \forall k \ge 0.
$$
\end{proposition}
\begin{proof}
We first remark that $\parallel M\parallel < 1$ implies the stability of the RSS scheme, since $M$ is the iteration matrix and $\rho(M)\le \parallel M\parallel $.

We take the difference and we let $w^{(k)}=u^{(k)}-v^{(k)}$. We have
$$
 \Frac{w^{(k+1)}-w^{(k)}}{\Delta t} +\tau B (w^{(k+1)}-w^{(k)}) +(\tau B-A)(v^{(k+1)} -v^{(k)})=-Aw^{(k)}
 $$
Hence, after the usual simplifications, we can write
$$
w^{(k+1)}=(Id-\Delta t(Id+\tau \Delta t B)^{-1}A)w^{(k)} -\Delta t (Id+\tau \Delta t B)^{-1}(\tau B-A)(v^{(k+1)} -v^{(k)}).
$$ 
Using the definition of $M$, we obtain directly the estimate
$$
\parallel w^{(k+1)}\parallel \le  \parallel M \parallel \parallel w^{(k+1)}\parallel 
+ \Delta t \parallel  (Id+\tau \Delta t B)^{-1}\parallel \parallel  \tau B-A\parallel  \parallel v^{(k+1)} -v^{(k)}\parallel .
$$
We first remark that we have the relations
$$
v^{(k+1)} -v^{(k)}=\Delta t (Id+\Delta t A)^{-1}(f-Av^{(k)}),
$$
and
$$
r^{(k+1)}=(Id-\Delta t A)(Id+\Delta t A)^{-1} r^{(k)},
$$
where we have set $r^{(k)}=f-Av^{(k)}$. It follows that 
$$
  v^{(k+1)} -v^{(k)} = \Delta t (Id+\Delta t A)^{-1}\left((Id-\Delta t A)(Id+\Delta t A)^{-1}\right)^{k}r^{(0)}.
$$
Then
$$
\parallel  v^{(k+1)} -v^{(k)} \parallel
\le \Delta t \parallel  (Id+\Delta t A)^{-1} \parallel  \parallel  ( Id-\Delta t A)(Id+\Delta t A)^{-1}\parallel^k
\parallel  r^{(0)} \parallel .
$$
We set $\gamma=\parallel   (Id-\Delta t A)(Id+\Delta t A)^{-1}\parallel$, we have of course $\gamma <1, \forall \Delta t >0$. 
A simple induction gives
$$
\parallel w^{(k+1)}\parallel \le  \parallel M \parallel^k \parallel w^{(0)}\parallel 
+  \Delta t ^2\parallel  (Id+\tau \Delta t B)^{-1}\parallel \parallel  \tau B-A\parallel \parallel  (Id+\Delta t A)^{-1} \parallel
\displaystyle{\sum_{j=0}^k\gamma^{(j)} \parallel M \parallel^{k-j}} \parallel  r^{(0)} \parallel .
$$
Using $ \parallel M \parallel<1$ and $w^{(0)}=0$ we find
$$
\parallel w^{(k+1)}\parallel \le  
+  \Delta t \parallel  (Id+\tau \Delta t B)^{-1}\parallel \parallel  \tau B-A\parallel \Delta t \parallel  (Id+\Delta t A)^{-1} \parallel
\displaystyle{\sum_{j=0}^k\gamma^{(j)}\parallel  r^{(0)}} \parallel .
$$
Hence
$$
\parallel w^{(k+1)}\parallel \le  \Delta t^2 \parallel  \tau B-A\parallel \Frac{1}{1-\gamma}\parallel  r^{(0)} \parallel ,
$$
which shows the dependence on $\parallel  \tau B-A\parallel$.
Of course if $\tau=1$ and $B=A$ we have $w^{(k)}=0, \forall k$.
Hence the result.
\end{proof}
\\

We deduce immediately the
\begin{corollary}
The RSS method is first order accurate.
\end{corollary}
\begin{proof}
It suffices to own that the Backward Euler method is first order accurate.
\end{proof}
\\
\\
We now will consider the particular  the case  in which $B$ is a diagonal matrix. This choice allows a fast solution of the implicit part of RSS, the matrix $Id+\tau \Delta t B$ being also diagonal. In general situations, a diagonal preconditioner is not the most efficient, but in some cases, e.g. when the discrete problem in written in hierarchical-like bases, it is particularly interesting to consider $B$ as a diagonal matrix: A. Cohen {\it et al} \cite{AverbuchCohenIsraeli} have introduced the RSS scheme for a problem discretized in a wavelet basis, in which the diagonal part of the stiffness matrix is an inconditional preconditioner; this was independently applied by Costa {\it et al} \cite{BCostaPHD,CDGT} using Fourier and Chebyshev expansions and then in finite differences with incrementals unkowns by Chehab and Costa \cite{ChehabCosta1,ChehabCosta2,ccraptech}. The underlying idea is to decompose the unknowns $U$ of the nodal basis into a hierarchy of arrays of details at different levels $({\hat U}_0, {\hat U}_1,\cdots,{\hat U}_d)^T$; here ${\hat U}_0$ is associated to a coarse discretization and then captures only low frequencies
while the other set of components ${\hat U}_j, \ j=1,\cdots, d$ are details associated to refinment of the coarse approximation space and capture high frequencies. \\
The time limitation of the explicit schemes for parabolic problems depends on its capability to contain high frequencies expansions. In Fourier-like basis, the details attached to high frequencies are small quantities regardless to the details attached to low frequencies since they contribute residually to the energy norm of the signal. As proposed
in \cite{CDGT,ChehabCosta1,ChehabCosta2} this situation allows to damp differently the high and the low frequencies components without deterioring the consistency of the scheme.\\

Consider for instance  the heat equation
$$
u_t-u_{xx}=f,
$$
that we discretize in time with the forward Euler scheme as
$$
\Frac{u^{(k+1)}-u^{(k)}}{\Delta t} -u^{(k)}_{xx}=f.
$$
In \cite{BCostaPHD,CDGT},  this scheme was proposed to be stabilized as
$$
\Frac{u^{(k+1)}-u^{(k)}}{\Delta t} +\beta
\Frac{u^{(k+1)}-u^{(k)}}{\Delta t} -u^{(k)}_{xx}=f.
$$
Considering a Fourier discretization, we find, after simplifications
$$
\Frac{{\hat u}^{(k+1)}_m-{\hat u}_m^{(k)}}{\Delta t} -\Frac{\beta}{\Delta t}({\hat u}_m^{(k+1)}-{\hat u}_m^{(k)})
=-k^2{\hat u}^{(k)}_m + {\hat f}_m, \ m=1,\cdots, N.
$$
If we decompose the frequency range $[1,N]$ into $\displaystyle{\cup_{j=1}^{d}[N_{j-1},N_{j}]}$, the above scheme can be applied 
to each range with a different stabilizing parameter $\beta_j$, so we obtain, for $j=1,\cdots,d$
$$
\Frac{{\hat u}^{(k+1)}_m-{\hat u}^{(k)}_m}{\Delta t} -\Frac{\beta_j}{\Delta t}({\hat u}^{(k+1)}_m-{\hat u}^{(k)}_m)
=-m^2{\hat u}^{(k)}_m + {\hat f}_m, \ m=N_{j-1}, \cdots, N_{j}.
$$
Letting ${\tilde \beta}_i=\Frac{\beta_i}{\Delta t}$, the last scheme is rewritten as
$$
{\hat u}^{(k+1)}_m=\left(1-\Frac{m^2\Delta t}{1+\Delta t {\tilde \beta_i}}\right){\hat u}^{(k)}_m
+\Frac{\Delta t}{1+\Delta t {\tilde \beta_i}}{\hat f}_m, \ m=N_{j-1}, \cdots, N_{j}.
$$
The stability condition is then
$$
(m^2-2{\tilde \beta}_j)\Delta t < 2, m=N_{j-1}, \cdots, N_{j}.
$$
The scheme is unconditionally stable at level $j$ if ${\tilde \beta}_j\ge \Frac{N^2_{j+1}}{2}$ and stable under condition $0<\Delta t <\Frac{2}{N^2_{j+1}-2{\tilde \beta}_j}$ otherwise. This condition is of course similar to that found in Proposition \ref{RSS_Stab_lin}.

Now, considering all the components, we write
$$
\Frac{{\hat u}^{(k+1)}-{\hat u}^{(k)}}{\Delta t} -B({\hat u}^{(k+1)}-{\hat u}^{(k)})
=-D{\hat u}^{(k)} + {\hat f}, \
$$
with $D=diag(1,4,9, \cdots N^2)$ and 
$$
B=
\left(
\begin{array}{llll}
B_1 & 0 & & 0\\
0 & \ddots &  & \\
 & & \ddots &   \\
 0 & &  & B_d\\
\end{array}
\right),
$$
where $B_j=\beta_j Id_j$, $Id_j$ being the identity matrix of size $(N_{j}-N_{j-1}) \times (N_{j}- N_{j-1})$ (we have set $N_0=0$ and $N_d=N$).

$B$ is then a preconditioner of $D$ in the Fourier space; in the linear case, these two approaches coincide.
Costa and Chehab \cite{ChehabCosta1,ChehabCosta2} have
extended this scheme to hierarchical discretizations in finite
differences.\\
Note that this framework allows to damp high frequencies and to leave unchanged the low ones changing only slightly the consistency of the scheme while increasing its stability. This can be done typically by taking $\beta_j=0$ on the low freqencies components and $\beta_j >>1$ on the high ones.

As stated in the introduction, we concentrate on problems discretized in the nodal basis, however, it is useful to make the link with the hierarchical approach. We give a block version of Proposition \ref{RSS_Stab_lin}.

We assume that the stiffness matrix $A$ written in a detail basis posseses the following block decomposition
$$
A=\left(
\begin{array}{cccc}
A_{1,1} & A_{1,2}&\cdots & A_{d,d}\\
A_{2,1} &  & & \\
\vdots & &  &\\
 A_{d,1}& & \cdots &A_{d,d}\\
\end{array}
\right).
$$
We note the corresponding  block decomposition of a vector $U$ as $U=(u_1,\cdots, u_d)^T$.
We have the
\begin{theorem}
A sufficient stability condition is
$$
\Frac{1}{\Delta t}   +\tau \beta_i -\Frac{1}{4}\left(\sum_{j=1}^d\parallel A_{i,j}\parallel  +\parallel A_{j,i}\parallel\right)
>0, i=1,\cdots, d.
$$
\end{theorem}
\begin{proof}
Taking the scalar product of the equation with $U^{(k+1)}-U^{(k)}$, we find
$$
\begin{array}{ll}
\Frac{1}{\Delta t} \parallel U^{(k+1)}-U^{(k)} \parallel^2
+\tau <B(U^{(k+1)}-U^{(k)}),U^{(k+1)}-U^{(k)}> & \\
&\\
-\Frac{1}{2}\left(<AU^{(k)},U^{(k)}> -<AU^{(k+1)},U^{(k+1)}>+
<A(U^{(k+1)}-U^{(k)}),U^{(k+1)}-U^{(k)}>\right)&=0,\\
\end{array}
$$
and using the block decomposition
$$
\begin{array}{ll}
\displaystyle{\sum_{i=1}^d\Frac{1}{\Delta t} \parallel u_i^{(k+1)}-u_i^{(k)} \parallel^2
+\tau \beta_i \parallel u_i^{(k+1)}-u_i^{(k)} \parallel^2} & \\
&\\
-\Frac{1}{2}\left(<AU^{(k)},U^{(k)}> -<AU^{(k+1)},U^{(k+1)}>\right)+
-\Frac{1}{2}\displaystyle{\sum_{i=1}^d\sum_{j=1}^d<A_{i,j}(u_j^{(k+1)}-u_j^{(k)}),u_i^{(k+1)}-u_i^{(k)}>}&=0,\\
\end{array}
$$
A sufficient condition for the energy stability is then
$$
\displaystyle{\sum_{i=1}^d\Frac{1}{\Delta t} \parallel u_i^{(k+1)}-u_i^{(k)} \parallel^2
+\tau \beta_i \parallel u_i^{(k+1)}-u_i^{(k)} \parallel^2}
-\Frac{1}{2}\displaystyle{\sum_{i=1}^d\sum_{j=1}^d<A_{i,j}(u_j^{(k+1)}-u_j^{(k)}),u_i^{(k+1)}-u_i^{(k)}>}>0.
$$
Since,
$$
\begin{array}{ll}
-\Frac{1}{2}
\displaystyle{\sum_{i=1}^d\sum_{j=1}^d<A_{i,j}(u_j^{(k+1)}-u_j^{(k)}),u_i^{(k+1)}-u_i^{(k)}>}
&\ge 
-\Frac{1}{2}
\displaystyle{\sum_{i=1}^d\sum_{j=1}^d\parallel A_{i,j}\parallel \mid u_j^{(k+1)}-u_j^{(k)}\mid \mid u_i^{(k+1)}-u_i^{(k)}\mid } \\
&\ge 
-\Frac{1}{4}
\displaystyle{\sum_{i=1}^d\sum_{j=1}^d\parallel A_{i,j}\parallel \mid u_j^{(k+1)}-u_j^{(k)}\mid^2}\\
&-\Frac{1}{4}\displaystyle{\sum_{i=1}^d\sum_{j=1}^d\parallel A_{i,j}\parallel \mid u_i^{(k+1)}-u_i^{(k)}\mid^2} \\
&=-\Frac{1}{4}
\displaystyle{\sum_{i=1}^d\left(\sum_{j=1}^d\parallel A_{i,j}\parallel  +\parallel A_{j,i}\parallel \right)}
\end{array}
$$
We find as sufficient stability condition
$$
\displaystyle{\sum_{i=1}^d\left(\Frac{1}{\Delta t}   +\tau \beta_i -\Frac{1}{4}\left(\sum_{j=1}^d\parallel A_{i,j}\parallel  +\parallel A_{i,j}\parallel\right)\right)\parallel u_i^{(k+1)}-u_i^{(k)} \parallel^2} >0.
$$
In particular, if $\tau \beta_i \ge \Frac{1}{4}\left(\sum_{j=1}^d\parallel A_{i,j}\parallel  +\parallel A_{i,j}\parallel\right)
>0$, the stability is unconditional.
\end{proof}
\\

It has to be noted that when $d=1$ and $B=Id$, we recover the result given by Proposition \ref{RSS_Stab_lin}.\\

We infer also that, the extradiagonal coefficients of stiffness matrices in  hierarchical-like basis enjoy of a descreasing magnitude properties far from the diagonal making successful the approach.

If we consider Fourier basis, the stifness matrix $A$ is diagonal and the stability condition at range $j$ is
$$
(m^2-2{\tilde \beta}_j)\Delta t < 2, m=N_{j-1}, \cdots, N_{j}
$$
so, taking $\beta_j=\Frac{N^2_{j+1}}{2}$, we have an incondiditionally stable RSS scheme.


\subsubsection{The nonlinear case}
We aim at applying the RSS Scheme to reaction-diffusion equation (such as Allen-Cahn's, see section 4) say
\begin{eqnarray}
\Frac{\partial u}{\partial t} -\Delta u +\Frac{1}{\epsilon^2}f(u)=0, & x\in \Omega, t>0,\\
\Frac{\partial u}{\partial n}= 0 & \partial \Omega , t>,\\
u(x,0)=u_0(x) & x\in \Omega .
\end{eqnarray}
 The RSS scheme applied to the discretized scheme reads as
\begin{eqnarray}
 \Frac{u^{(k+1)}-u^{(k)}}{\Delta t} +\tau B (u^{(k+1)}-u^{(k)}) +\Frac{1}{\epsilon^2}f(u^{(k)}) =-Au^{(k)}
 \label{RSSAC}
 \end{eqnarray}
 We set $E(u)=\Frac{1}{2}<Au,u>+\Frac{1}{\epsilon^2}<F(u),1>$, where $F$ is a primitive of $f$ that we choose such that $F(0)=0$. We say that the scheme is energy descreasing if
 $$
 E(u^{(k+1)}) < E(u^{(k)})
 $$
 Particularly,  if $F$ as nonnegative values, the scheme will be stable for the norm $\mid u \mid_A=\sqrt{<Au,u>}$.
\begin{theorem}
Assume that $f$ is ${\cal C}^1$ and $\mid f'\mid_{\infty}\le L$. We have the following stability conditions
\begin{itemize}
\item If $\tau\ge \Frac{\beta}{2}$ then
\begin{itemize}
\item if $\left(\Frac{\tau}{\beta}-\Frac{1}{2}\right)\lambda_{min} -\Frac{L}{2\epsilon^2}\ge 0$ then the scheme is unconditionally stable
\item if $\left(\Frac{\tau}{\beta}-\Frac{1}{2}\right)\lambda_{min} -\Frac{L}{2\epsilon^2}< 0$ then the scheme is table for
$$
0<\Delta t <\Frac{1}{\Frac{L}{2\epsilon^2} -\left(\Frac{\tau}{\beta}-\Frac{1}{2}\right)\lambda_{min}}
$$
\end{itemize}
\item If $\tau < \Frac{\beta}{2}$ then the scheme is table for
$$
0<\Delta t <\Frac{1}{\Frac{L}{2\epsilon^2} -\left(\Frac{\tau}{\beta}-\Frac{1}{2}\right)\rho(A)}
$$
\end{itemize} 
\label{Stab_AC}
\end{theorem}
\begin{proof}
We have, taking the scalar product in $\R^N$  of each element of the with $u^{(k+1)}-u^{(k)}$:
$$
\begin{array}{ll}
\Frac{1}{\Delta t} \parallel u^{(k+1)}-u^{(k)} \parallel^2
+\tau <B(u^{(k+1)}-u^{(k)}),u^{(k+1)}-u^{(k)}> =& -\Frac{1}{\epsilon^2}<f(u^{(k)},u^{(k+1)}-u^{(k)}> \\
&-<Au^{(k)},u^{(k+1)}-u^{(k)}>.\\
\end{array}
$$
Let us now consider the nonlinear term.  We have using Taylor-Lagrange's expansion
$$
F(u^{(k+1)}_i)-F(u^{(k)}_i)=f(u^{(k)}_i) (u^{(k+1)}_i-u^{(k)}_i) +\Frac{1}{2}f'(\xi_i) (u^{(k+1)}_i-u^{(k)}_i)^2.
$$
Taking the sum of these term for $i=1,\cdots, N$, we obtain
$$
<F(u^{(k+1)},1>-<F(u^{(k)},1>=<f(u^{(k)},u^{(k+1)}-u^{(k)}>+\Frac{1}{2}
\displaystyle{\sum_{i=1}^N f'(\xi_i) (u^{(k+1)}_i-u^{(k)}_i) ^2}.
$$
so that
$$
-\Frac{1}{\epsilon^2}<f(u^{(k)},u^{(k+1)}-u^{(k)}>
\le -\Frac{1}{\epsilon^2}(<F(u^{(k+1)},1>-<F(u^{(k)},1>)  
+\Frac{L}{2\epsilon^2} \parallel u^{(k+1)}-u^{(k)} \parallel^2.
$$
The other terms are treated exactly as in the linear case, and we use the identity
$$
<Au^{(k)},u^{(k+1)}-u^{(k)}>=\Frac{1}{2}\left( <Au^{(k+1)},u^{(k+1)}>-<A(u^{(k)},u^{(k)}>+<A(u^{(k+1)}-u^{(k)}),u^{(k+1)}-u^{(k)})>\right).
$$
Making use of these results, we find, after the usual simplifications
$$
\left(\Frac{1}{\Delta t}-\Frac{L}{2\epsilon^2}\right)\parallel u^{(k+1)}-u^{(k)} \parallel^2
+\left( \Frac{\tau}{\beta}-\Frac{1}{2}\right)<A(u^{(k+1)}-u^{(k)}),u^{(k+1)}-u^{(k)})>
+E(u^{(k+1)})-E(u^{(k)})\le 0
$$
Hence the stability is obtained when
$$
\left(\Frac{1}{\Delta t}-\Frac{L}{2\epsilon^2}\right)\parallel u^{(k+1)}-u^{(k)} \parallel^2
+\left( \Frac{\tau}{\beta}-\Frac{1}{2}\right)<A(u^{(k+1)}-u^{(k)}),u^{(k+1)}-u^{(k)})> >0.
$$
Hence the result.
\end{proof}

We point ou that in the linear case ($f=0$ and $L=0$), we recover the stability conditions given by poposition \ref{RSS_Stab_lin}. 

As an application, we can consider the simualtion of the Allen-Cahn equation
\begin{eqnarray}
\Frac{\partial u}{\partial t} -\Delta u+\Frac{1}{\epsilon^2}f(u)=0 & x\in \Omega, t >0,\\
\Frac{\partial u}{\partial n}=0& x \in \partial \Omega , \forall t >0,\\
u(x,0)=u_0(x) & x\in \Omega.
\end{eqnarray}
\\
This reaction-diffusion equation describes the process of phase separation in  many situations, \cite{AllenCahn}.
In practice, we will chose $f(u)=u(u^2-1)$. 

Notice that Shen {\it et al }\cite{JShenACCH} have proposed the scheme
 \begin{equation}
 \Frac{u^{(k+1)}-u^{(k)}}{\Delta t} +\Frac{S}{\epsilon^2} (u^{(k+1)}-u^{(k)}) +Au^{(k+1)}+\Frac{1}{\epsilon^2}f(u^{(k)}) =0.
 \label{ShenAC}
 \end{equation}

With Theorem \ref{Stab_AC}, we recover the stability conditions proposed by J. Shen when $A=B$ and $\tau=1$ and $S=0$ ; The term $\Frac{S}{\epsilon^2} (u^{(k+1)}-u^{(k)})$ plays the role of the stabilizator, following the same principle as the schemes introduced in \cite{BCostaPHD,CDGT}. The time restriction become harder when $\epsilon$
takes small values. This situation motivates the use of stabilized schemes.
Now, if $S>\Frac{L}{2}$, the scheme (\ref{ShenAC}) is unconditionally stable. This is to be compared with the RSS scheme (\ref{RSSAC}). We find as inconditional stability condition
$$
\tau > \Frac{\beta L}{2\lambda_{min}}+\Frac{\beta\epsilon^2}{2},
$$
which is a comparable condition for $\epsilon$ small enough and $\beta$ bounded, since in pratice $\lambda_{min}$ is a positive constant which not depend on the dimension of the problem,  the first eigenvalue of the stifness matrix is indeed nicely captured by the discretization schemes.  However the additional stabilizing terms can deteriorate the consistency. 
%
\subsection{Richardson extrapolation}
In \cite{AverbuchCohenIsraeli} the authors have proposed to increase
the accuracy of the stabilized scheme by smoothing the residual
using a Richardson extrapolation process:\\
The solution of
\begin{eqnarray*}
\Frac{d u }{dt}=F(u),
\end{eqnarray*}
by the forward Euler scheme defines the iterations
\begin{eqnarray*}
u^{k+1}=u^{k}+\Delta t F(u^k)=G_{\Delta t}(u^k).
\end{eqnarray*}
The smoothed sequence is defined by
\begin{eqnarray*}
v_1=G_{\Delta t}(u^k),\\
v_{2,0}=G_{\Delta t /2}(u^k),\\
v_{2,1}=G_{\Delta t /2}(v_{2,0}),\\
u^{k+1}=2v_{2,1}-v_1.
\end{eqnarray*}
It is second order accurate in time. The accuracy of the stabilized
scheme is increased by applying the extrapolation to the iteration
operator
$$
G_{\Delta t,\tau}(u^k)=u^k +\Delta t (Id +\tau \Delta t
B)^{-1}F(u^k).
$$
The improvement is studied analytically in \cite{RibotSchatzman2}, but
numerical evidences point out the efficiency of the approach, see
also the numerical results presented in Section 4.

Below the Extrapoled RSS scheme\\
 \begin{center}
\begin{minipage}[H]{12cm}
  \begin{algorithm}[H]
    \caption{: Extrapoled RSS Scheme}\label{ExtraRSS}
    \begin{algorithmic}[1]
        \State $u^{(0)}$ given\\
            \For{$k=0,1, \cdots$until convergence}
             \State {\bf Solve} $ (Id+\tau \frac{\Delta t}{2}B) v_1=-\frac{\Delta t}{2} F(u^{(k})$
              \State {\bf Set} $u_1=u^{(n)} +v_1$
               \State {\bf Solve} $ (Id+\tau \frac{\Delta t}{2}B) v_2=-\frac{\Delta t}{2} F(u_1)$
              \State {\bf Set}  $u_2=u_1+v_2$
               \State {\bf Solve} $(Id+\tau \Delta tB) v_3=-\Delta t  F(u^{(k)})$
               \State {\bf Set} $u_3=u^{(n)}+v_3$
               \State  {\bf Set}  $u^{(k+1)}=2u_2-u_3$              
            \EndFor
    \end{algorithmic}
    \end{algorithm}
\end{minipage}
\end{center}
\section{Discretisation in space and preconditioning}
\subsection{Compact FD Scheme Discretization}
A way to obtain a high level of accuracy with a finite difference scheme, that can be compared with the spectral one, is to implement finite difference compact schemes \cite{Lele}. These schemes consist in approaching a linear operator (differentiation, interpolation) by a rational (instead of polynomial-like) finite differences scheme. 
We describre briefly here only their construction for the approximation of the first and the second derivative and we refer to \cite{Lele} for more details. We first consider the schemes in space dimension one.

Let $U=(U_1,\cdots,U_n)^T$ denotes a vector whose the components are the approximations of a regular function $u$ at (regularly spaced) grid points $x_i=ih$, $i=1,\cdots, n$.  We compute approximations of $V_i={\cal L}(u)(x_i)$ as solution of a system
$$
P . V= Q U,
$$
so the approximation matrix is formally $B=P^{-1}Q$.
When $P=Id$, the scheme is explicit and we recover the framework of classical finite difference schemes;  when $P\neq Id$, the scheme is implicit an dit is possible to reach high order accuracy, the implicity allows to mimmic the spectral global dependence. In practice, $P$ is a banded sparse matrix easy to invert and very well conditioned making the compact scheme approach robust and not costly regardless to the precision brought.
We here give the matrices $P$ and $Q$ for the fourth order approximation of the first and the second derivative, details can be found in \cite{Lele}.

Let us begin with the first derivative. We have

$$P=\begin{pmatrix}
1 & \frac{1}{4} &   &   \\ 
\frac{1}{4} & 1 & \ddots &   \\ 
  & \ddots & \ddots & \frac{1}{4} \\ 
  &   & \frac{1}{4} & 1
\end{pmatrix}, $$

$$Q=\dfrac{1}{2h} \begin{pmatrix}
a_1 & a_2 & a_3 & a_4 &   \\ 
-\frac{3}{2} & 0 & \frac{3}{2} &   &   \\ 
  & \ddots & \ddots & \ddots &   \\ 
  &   & -\frac{3}{2} & 0 & \frac{3}{2} \\ 
  & -a_4 & -a_3 & -a_2 & -a_1
\end{pmatrix}, $$

with $a_1=-2$, $a_2=3$, $a_3=-\frac{2}{3}$ and $a_4=\frac{1}{8}$.
\\
\\
In the same way, we can build the fourth order compact schemes for the second order derivative. We first
consider the compact scheme asociated to the discretization of the second derivative with homogeneous Dirichlet boundary conditions. We have

with
$$ P= \begin{pmatrix}
1 & \frac{1}{10} &   &   &   \\ 
\frac{1}{10} & 1 & \frac{1}{10} &   &   \\ 
  & \ddots & \ddots & \ddots &   \\ 
  &  & \frac{1}{10} & 1 & \frac{1}{10} \\ 
 &  &  & \frac{1}{10} & 1
\end{pmatrix}, $$

and $$ Q = \frac{1}{h^2} \begin{pmatrix}
a_1 & a_2 & a_3 & a_4 & a_5 &   &   \\ 
-\frac{6}{5} & \frac{12}{5} & -\frac{6}{5} &   &   &   &   \\ 
  & -\frac{6}{5} & \frac{12}{5} & -\frac{6}{5} &   &   &   \\ 
  &   & \ddots & \ddots & \ddots &   &   \\ 
  &   &   & -\frac{6}{5} & \frac{12}{5} & -\frac{6}{5} &   \\ 
  &   &   &   & -\frac{6}{5} & \frac{12}{5} & -\frac{6}{5} \\ 
  &   & a_{N-4} & a_{N-3} & a_{N-2} & a_{N-1} & a_N
\end{pmatrix}, $$
here the constant  $a_1$, $a_2$, $a_3$, ... are given by

$$ \left\lbrace \begin{array}{rcccl}
a_1&=&a_{N}&=&-\frac{67}{60},   \\
a_2&=&a_{N-1}&=&-\frac{7}{12},   \\
a_3&=&a_{N-2}&=&\frac{13}{10},   \\
a_4&=&a_{N-3}&=&-\frac{61}{120},   \\
a_5&=&a_{N-4}&=&\frac{1}{12}.   \\
\end{array}  \right.  $$

Now applying the same approach, we can consider fourth order compact schemes for the second derivative with associated homogeneous Neumann Boundary conditions

$$ M= \begin{pmatrix}
1 & \frac{1}{10} &   &   &   \\ 
\frac{1}{10} & 1 & \frac{1}{10} &   &   \\ 
  & \ddots & \ddots & \ddots &   \\ 
  &  & \frac{1}{10} & 1 & \frac{1}{10} \\ 
 &  &  & \frac{1}{10} & 1
\end{pmatrix}, $$ 

and $$ N = \frac{1}{h^2} \begin{pmatrix}
a_1 & a_2 & a_3 & a_4 & a_5 &   &   \\ 
-\frac{6}{5} & \frac{12}{5} & -\frac{6}{5} &   &   &   &   \\ 
  & -\frac{6}{5} & \frac{12}{5} & -\frac{6}{5} &   &   &   \\ 
  &   & \ddots & \ddots & \ddots &   &   \\ 
  &   &   & -\frac{6}{5} & \frac{12}{5} & -\frac{6}{5} &   \\ 
  &   &   &   & -\frac{6}{5} & \frac{12}{5} & -\frac{6}{5} \\ 
  &   & a_{N-4} & a_{N-3} & a_{N-2} & a_{N-1} & a_N
\end{pmatrix}, $$

with

$$ \left\lbrace \begin{array}{rcccl}
a_1&=&a_{N}&=&\frac{2681}{480},   \\
a_2&=&a_{N-1}&=&-\frac{32}{3},   \\
a_3&=&a_{N-2}&=&\frac{113}{40},   \\
a_4&=&a_{N-3}&=&-\frac{13}{15},   \\
a_5&=&a_{N-4}&=&\frac{59}{480}.   \\
\end{array}  \right.  $$

To obtain the compact schemes of first and second order derivative in space dimension 2 and 3, it suffices to use 
the previous schemes and to expand them tensorally.\\
The finite differences discretization matrices of derivative on cartesian domains can be obtained by those on the interval using Kronecker products. Indeed, if $A^N_{xx}$ denotes the discretization matrix on $[0,1]$ associated to Dirichlet Boundary conditions, using $N$ internal discretization points, then 
$$
Id_M\otimes A^N_{xx}
$$
will be the discretization matrix of the same operator but on a grid composed of $N\times M$ points, the corresponding laplacian matrix will be $A^M_{xx}\otimes Id_N+Id_M\otimes A^N_{xx}$. We recall that the Kronecker product of a $m\times n$ matrix $A$ by a $p\times q$ matix $B$ is defined as
$$
A \otimes B= \begin{pmatrix}
a_{11} B&\cdots &  a_{1n} B \\ 
\vdots& \ddots&\vdots\\
 a_{m1} B&\cdots &  a_{mn} B \\ 
\end{pmatrix}. $$
\subsection{Preconditioning FD Compact schemes using second order FD}
The matrices associated to compact finite difference schemes are full, this is due to the implicit nature of the scheme.
A natural idea to built a sparse preconditioner is to use the matrix obtained by applying a lower accurate finite discretization scheme, particularly a second order one. We here describe the approach for one dimensional problem, extensions to higher dimensions are obtained using kronecker products, also we restrict to linear problem for simplicity. Let $A_2$ (resp $A_4$) be the second order (resp. the fourth order) discretization matrice  of $-\Delta$ on a regular grid composed of $N$ internal points. The RSS scheme writes as
\begin{equation}
\Frac{u^{(k+1)}-u^{(k)}}{\Delta t}+\tau A_2(u^{(k+1)}-u^{(k)})+A_4u^{(k)}=f.
\end{equation}
\\

The numerical treatment of non homogeneous (possibly time depending) Dirichlet boundary conditions can be realized with the RSS approach.  Indeed, let us note $A_m(u,n)$, $m=2,4$, the m$th$ order finite difference discretization of $-\Delta$ of $u$  with Dirichlet conditions at time $n\Delta t$, note that this operator is affine. The stabilized scheme writes  formally as
\begin{equation}
\Frac{u^{(k+1)}-u^{(k)}}{\Delta t}+\tau
(A_2(u^{(k+1)},k+1)-A_2(u^{(k)},k))+A_4(u^{(k)},k)=f,
\end{equation}
Making the approximation $A_2(u^{(k+1)},k+1)\simeq A_2(u^{(k+1)},k)$, we obtain
\begin{equation}
\Frac{u^{(k+1)}-u^{(k)}}{\Delta t}+\tau
A_2(u^{(k+1)}-u^{(k)})+A_4(u^{(k)},k)=f.
\end{equation}

It is to be pointed out that for the solutions of 2D and 3D Poisson problems,  the number of iterations to  the convergence is not dependent on the dimension of the system. Also, in the case of Poisson-like problem, we can use FFT to solve the preconditioning system, speeding up the resolution of the original linear system. This approach will be followed also  for the fast solution of the heat equation that arises in parabolic problems.
\begin{remark}
A analogous approach have been done in the context of hierarchical preconditioners in finite differences, where the fourth order discretization matrix of $-\Delta$ was preconditioned by the hierarchical transfert matrix attached to the second order accurate discretization of $-\Delta$, see \cite{ChehabCS}.
\end{remark}

Below, we report numerical results on the solution of 2D and 3D Poisson problems when discretized by fourth order compact schemes. The preconditioning matrix is obtained by applying  corresponding second order finite difference schemes. We took a random r.h.s {\tt b=1-2*rand(N,1)} so that many frequencies including high ones are present. The initial data is $u=0$, the tolerance parameter $10^{-12}$. The result we report is the maximum number of external iterations to convergence, on 5 independent numerical resolutions, the number of discretization point per direction $n$  is reported as $(n)$.   The stiffnes matrices are then of respective sizes $n^2\times n^2$ (2D problem) and
$n^3\times n^3$ (3D problem).

\begin{table}[!h]
\begin{center}

\begin{tabular}{|c||c|c|c|c|c|c|}
\hline 
 Problem & $\#$ it. (n) &  $\#$ it. (n) ) & $\#$ it. (n) & $\#$ it. (n) & $\#$it. (n) & $\#$it. (n)  \\
 \hline
 Poisson 2D & 12 (n=15) &  11  (n=31) &  10 (n=63)  & 10 (n=127) & 9 (n=255) & 8 (n=511)\\
 \hline
 Poisson 3D & 12 (n=15) &  11  (n=31) &  11 (n=63) & & &\\
\hline 
\end{tabular} 
\caption{Solutions of 2D and 3D Poisson problem with GMRES and second order preconditioner}
\label{PrecondPoisson}
\end{center}
\end{table}

Here, the fourth order discretization matrix $A$ of $-\Delta$ is nonsymmetric while the preconditioning matrix $B$ is.
However, in practice the RSS method is still efficient. This is due to the small size of the skewsymmetric part of $A$. Indeed, denoting by $\delta= \parallel A-A^T\parallel$, we can prove the following general stability result.
\begin{theorem}
Let $A \in {\cal M}_{n}({\mathbb R}^N)$. We assume that $A$ is positive definite and $B$ a symmetric definite positive preconditioning matrix of $A$ satisfy hypothesis ${\cal H}$. We set
$\delta= \parallel A-A^T\parallel$ and $\Phi(\xi)= (\beta^2-2\alpha \tau)\xi + \Frac{1}{4\xi}\delta^2$.
Assume that $\Frac{\beta^2}{2\alpha}-\Frac{\delta^2}{8\alpha \lambda_{min}(B)^2} \ge 0$. Then the RSS scheme has the following stability conditions
\begin{itemize}
\item[i.] if $
\tau \ge \Frac{\beta^2}{2\alpha}+\Frac{\delta^2}{8\alpha \lambda^2_{min}(B)} \ge \Frac{\beta^2}{2\alpha}.
$
then the scheme is unconditionally stable.
\item [ii.] If $\tau \le \Frac{\beta^2}{2\alpha}-\Frac{\delta^2}{8\alpha \lambda_{max}(B)^2}$ then the scheme is table under condition
$$
0<\Delta t < \Frac{2\alpha}{\Phi(\lambda_{max}(B))}
$$
\item[iii.] If $\Frac{\beta^2}{2\alpha}-\Frac{\delta^2}{8\alpha \lambda_{max}(B)^2}\le \tau < 
\Frac{\beta^2}{2\alpha}+\Frac{\delta^2}{8\alpha \lambda_{min}(B)^2}$
then the scheme is table under condition
$$
0<\Delta t < \Frac{2\alpha}{\Phi(\lambda_{min}(B))}
$$
\item [iv.] If $
\Frac{\beta^2}{2\alpha}-\Frac{\delta^2}{8\alpha \lambda_{min}(B)^2}< \tau <\Frac{\beta^2}{2\alpha}-\Frac{\delta^2}{8\alpha \lambda_{max}(B)^2}
$
 then the scheme is table under condition
$$
0<\Delta t < \Frac{2\alpha}{\Max(\Phi(\lambda_{min}(B)),\Phi(\lambda_{max}(B)))}
$$
\end{itemize}
Here $\lambda_{min}(B)$ (resp.  $\lambda_{max}(B)$ denotes the lowest (resp. the largest) eigenvalue of $B$.
\label{theo_gen_stab}
\end{theorem}
\begin{proof}
The RSS scheme reads as
$$
u^{(k+1)}=u^{(k)} -\Delta t \left(Id + \tau \delta t B\right)^{-1}Au^{(k)}=Mu^{(k)} ,
$$
and is stable under the necessary and sufficient condition $\rho(M)<1$; in the general case the eigenvalues of $M$ can be complex.
Let $v\in {\mathbb C}^N$ be an eigenvector of $M$ associated to the eigenvalue $\lambda=a+ib$.
We have
$$
\lambda v=Mv,
\mbox{ so } 
(1-\lambda)\left(Id + \tau \delta t B\right)v=\Delta t  Av .
$$
We decompose $v$ as $v=v_1+iv_2$ and, separating real and imaginary parts, we obtain
$$
\begin{array}{ll}
Av_1&=\Frac{1-a}{\Delta t}\left( v_1+\tau \Delta t Bv_1\right) +
\Frac{b}{\Delta t}\left( v_2+\tau \Delta t Bv_2\right),\\
Av_1&=\Frac{1-a}{\Delta t}\left( v_2+\tau \Delta t Bv_2\right) -
\Frac{b}{\Delta t}\left( v_1+\tau \Delta t Bv_1\right).\\
\end{array}
$$
We have then, on the one hand
$$
<Av_1,v_1>+<Av_2,v_2>=\Frac{1-a}{\Delta t}\left(  \parallel v_1\parallel^2+\parallel v_2\parallel^2+\tau \Delta t <Bv_1,v_1>+\tau \Delta t <Bv_2,v_2>\right),
$$
and, on the other hand,
$$
<Av_1,v_2>-<Av_2,v_1>=
\Frac{b}{\Delta t}\left( \parallel v_1\parallel^2+\parallel v_2\parallel^2+\tau \Delta t <Bv_1,v_1>+\tau \Delta t <Bv_2,v_2>\right).
$$
We now set for convenience $N(v_1,v2)=\parallel v_1\parallel^2+\parallel v_2\parallel^2+\tau \Delta t <Bv_1,v_1>+\tau \Delta t <Bv_2,v_2>$. 
We infer from the previous identities
$$
a=1-\Frac{\Delta t }{N(v_1,v_2)} \left( <Av_1,v_1>+<Av_2,v_2> \right)
\mbox{ and }
b=\Frac{\Delta t }{N(v_1,v_2)}<(A-A^T)v_1,v_2>.
$$
The stability condition is
\begin{eqnarray}
\label{gen_stab}
\eta=a^2+b^2<1.
\end{eqnarray}
After the usual simplifications, we find  as a sufficient condition
$$
\begin{array}{ll}
\eta &\le 1-2\Frac{\Delta t}{N(v_1,v_2)}\alpha \left( <Bv_1,v_1>+<Bv_2,v_2>\right)\\
&+\Frac{\Delta t^2}{N(v_1,v_2)^2}\beta^2  \left( <Bv_1,v_1>+<Bv_2,v_2>\right)^2\\
&+\Frac{\Delta t^2}{4N(v_1,v_2)^2}\parallel A-A^T\parallel^2\left( \parallel v_1\parallel^2+ \parallel v_2\parallel^2\right)^2. \\
\end{array}
$$
We now let $Z=<Bv_1,v_1>+<Bv_2,v_2>$ and $Y=\parallel v_1\parallel^2+ \parallel v_2\parallel^2$. The last inequality reads
$$
\Delta t \left( (\beta^2-2\alpha \tau)Z^2 +\Frac{1}{4}\parallel A-A^T\parallel^2 Y^2\right) < 2\alpha YZ .
$$
At this point, we set $\xi=\Frac{Z}{Y}=\Frac{<Bv_1,v_1>+<Bv_2,v_2>}{\parallel v_1\parallel^2+ \parallel v_2\parallel^2}$. Note that we have $\lambda_{min}(B)\le \xi \le \lambda_{max}(B)$.
After usual simplifications, the sufficient stability condition (\ref{gen_stab}) writes now as
$$
\Delta t \left( (\beta^2-2\alpha \tau)\xi + \Frac{1}{4}\delta^2\Frac{1}{\xi}\right) < 2\alpha,
$$
where we have set $\delta =\parallel A-A^T\parallel$.
We use the function $\Phi(\xi)= (\beta^2-2\alpha \tau)\xi + \Frac{1}{4\xi}\delta^2$ ; $\Phi$ is obviously regular on $[\lambda_{min}(B),\lambda_{max}(B)]$, since $\lambda_{min}(B) >0$, $B$ is assumed to be SPD.
We deduce the following sufficient stability conditions
\begin{itemize}
\item if $\Phi(\xi) \le 0, \forall \xi \in [\lambda_{min}(B),\lambda_{max}(B)]$, the scheme is unconditionnaly stable
\item if there exists $\xi \in [\lambda_{min}(B),\lambda_{max}(B)]$ such that $\Phi(\xi) >0$, then a sufficient stablility condition is
$$
0<\Delta t<\Frac{2\alpha}{\displaystyle{\Max_{\xi \in [\lambda_{min}(B),\lambda_{max}(B)]}\Phi(\xi)}}.
$$
\end{itemize}
We conclude by studying the two cases.
\\
\\
\underline{Unconditional stablity }:
To have $\Phi(\xi) \le 0, \forall \xi \in [\lambda_{min}(B),\lambda_{max}(B)]$, we must have
$\Phi(\lambda_{min}(B))\le 0$ and $\Phi(\lambda_{max}(B))\le 0$, that is
$$
\tau \ge \Frac{\beta^2}{2\alpha}+\Frac{\delta^2}{8\alpha \lambda^2_{min}(B)} \ge \Frac{\beta^2}{2\alpha}.
$$
We now remark that $\Phi'(\xi)=(\beta^2-2\alpha \tau) -\Frac{\delta^2}{4\xi^2} \le 0 \iff 
\tau \ge \Frac{\beta^2}{2\alpha}-\Frac{\delta^2}{8\alpha \xi^2}, \forall \xi \in [\lambda_{min}(B),\lambda_{max}(B)]$ . This is satisfied under the previous hypothesis, hence the first statement $[i.]$ of the theorem.
\\
\\
\underline{Conditional stability} : The condition is $\Max(\Phi(x)) >0$. We distinguish two cases
\begin{itemize}
\item $\Phi$ is monotone.
\begin{itemize}
\item $\Phi'(\xi) \ge 0$  and $\Phi(\lambda_{max}(B))>0$ ie
$\tau < \Frac{\beta^2}{2\alpha}+\Frac{\delta^2}{8\alpha \lambda_{max}(B)^2}$ and $\tau \le \Frac{\beta^2}{2\alpha}-\Frac{\delta^2}{8\alpha \lambda_{max}(B)^2}$. A sufficent stability condition is then $\tau \le \Frac{\beta^2}{2\alpha}-\Frac{\delta^2}{8\alpha \lambda_{max}(B)^2}$ and
$$
0<\Delta t < \Frac{2\alpha}{\Phi(\lambda_{max})},
$$
\item $\Phi'(\xi) \le 0$ and $\Phi(\lambda_{min}(B))>0$ that is $\Frac{\beta^2}{2\alpha}-\Frac{\delta^2}{8\alpha \lambda_{max}(B)^2}\le \tau < 
\Frac{\beta^2}{2\alpha}+\Frac{\delta^2}{8\alpha \lambda_{min}(B)^2}$
and
$$
0<\Delta t < \Frac{2\alpha}{\Phi(\lambda_{min}(B))},
$$
\end{itemize}
\item $\Phi'(\xi)=0$ for $\xi\in  ]\lambda_{min}(B),\lambda_{max}(B)[$. This means that
$$
\Frac{\beta^2}{2\alpha}-\Frac{\delta^2}{8\alpha \lambda_{min}(B)^2}< \tau <\Frac{\beta^2}{2\alpha}-\Frac{\delta^2}{8\alpha \lambda_{max}(B)^2}.
$$
This implies that $\Phi'(\lambda_{min}(B))<0$ so $\Max(\Phi(\xi))$ is reached at $\lambda_{min}(B)$ or at
$\lambda_{max}(B)$ and the stability condition is
$$
0<\Delta t < \Frac{2\alpha}{\Max(\Phi(\lambda_{min}(B)),\Phi(\lambda_{max}(B)))}.
$$
\end{itemize}
\end{proof}
We point out that the symmetric part of matrix $A$ is dominant for small values of $\delta$ and that in this case the above stability result is to be compared with that of Proposition (\ref{RSS_Stab_lin}). Particularly inconditional stability condition is obtained for $\tau \ge \Frac{\beta^2}{2\alpha}  +\Frac{\delta^2}{8\alpha \lambda^2_{min}(B)} \simeq  \Frac{\beta}{2}$, for inconditional preconditioners $B$ as those presented above.
\section{Numerical Results}
\subsection{The problem and the numerical schemes}
As stated in the introduction, we here aim at computing the steady state of the so-called 2D driven cavity problem in a rectangular domain $\Omega$. 
The steady state will be reached by a pseudo-temporal method, and for that purpose we consider the stream function-vorticity formulation ($\omega-\psi$) (see \cite{PeyretTaylorBook,TemamNSE} and the references therein):
\begin{eqnarray}
\label{Navier_Stokes_psi}
\Frac{\partial \omega}{\partial t}-\Frac {1}{Re}\Delta \omega
 +\Frac {\partial \phi}{\partial y}
\Frac {\partial \omega}{\partial x} - \Frac {\partial \phi}{\partial
x}
\Frac {\partial \omega}{\partial y}=0, & \mbox{ in }  \Omega\\
\Delta \psi =\omega, & \mbox{ in }   \Omega \\
\omega(x,y,0)=\omega_0(x,y),
\end{eqnarray}
that we supplement with proper boundary conditions. We denote by ${\Gamma}_{i} \
\ i=1,..,4$ the sides of the unit square $\Omega$ as follows:
${\Gamma}_{1}$ is the lower horizontal side, ${\Gamma}_{3}$ is the
upper horizontal side, ${\Gamma}_{2}$ is the left vertical side, and
${\Gamma}_{4}$ is the right vertical side.

\begin{figure}[h]
\begin{center}
\includegraphics[height=4.in,width=4.in]{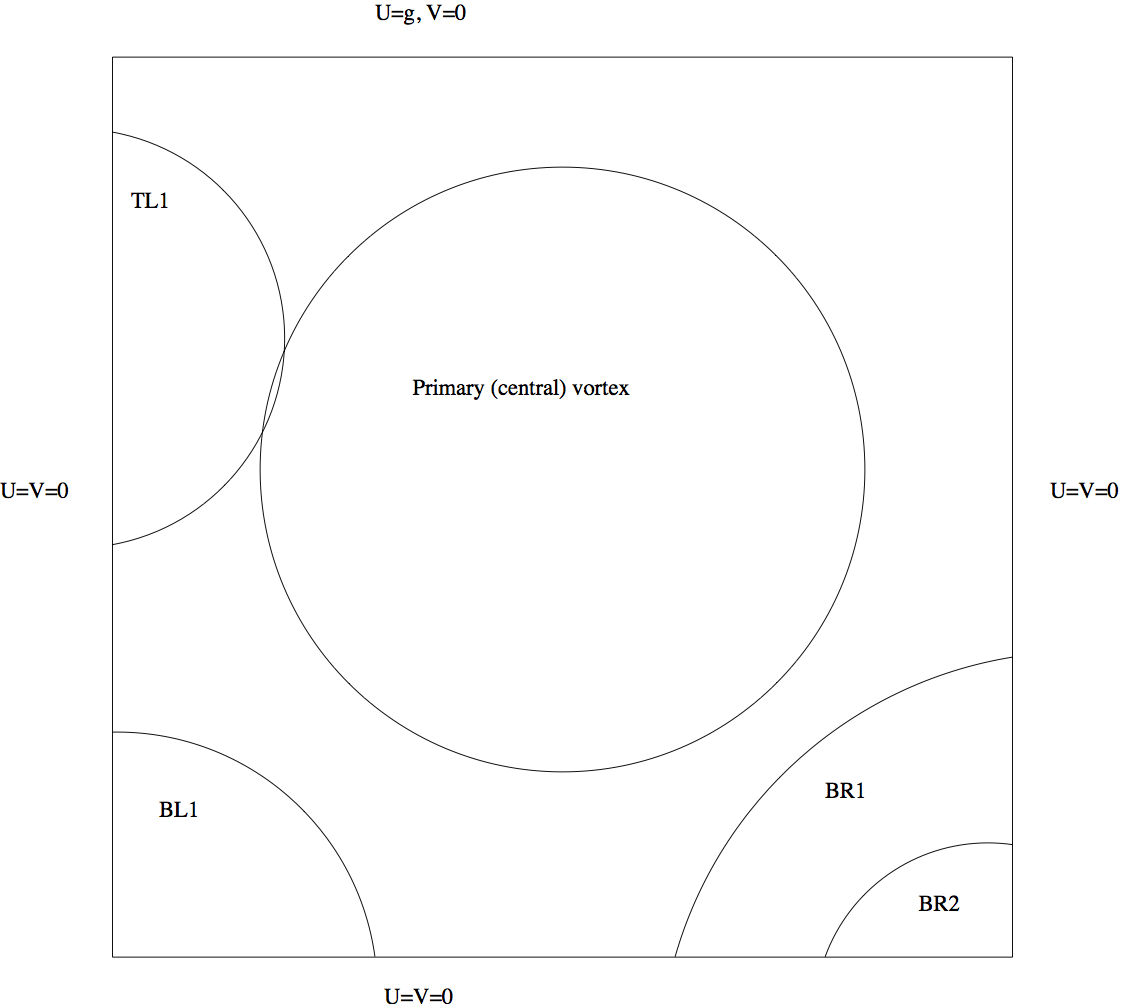}
\end{center}
\caption{The lid driven cavity - Schematic localization of the mean
vortex regions}
\end{figure}

We distinguish two different driven flows, according to the choice of
the boundary conditions on the velocity. More precisely we have
\begin{itemize}
\item $g(x)=1$ : Cavity A (lid driven cavity)
\item $g(x)=(1-(1-2x)^2)^2$ : Cavity B (regularized lid driven cavity)
\end{itemize}
We will also consider the steady  linear part of this equation (the Stokes problem), whose the solution will be chosen as the initial condition of (\ref{Navier_Stokes_psi})
\begin{eqnarray}
\label{Stokes_psi}
-\Frac {1}{Re}\Delta \omega=0, &\mbox{ in }  \Omega\\
\omega\mid_{\partial \Omega}=\Delta \psi \mid_{\partial \Omega} & \mbox{on} \partial \Omega,\\
\Delta \psi =\omega, &\mbox{ in } \Omega\\
\psi\mid_{\partial \Omega}=0 & \mbox{ on } \partial \Omega,
\end{eqnarray}
The RSS scheme is based on two different finite differences discretization of differential operators at the same grid points: a second order finite difference scheme will be used for preconditiong while the fourth order compact scheme is
implemented for the effective approximation to the solution.

The boundary conditions on $\omega$ are derived by the
discretization of $\Delta \psi$ on the boundaries. With the
conditions on $u$ and $v$ we have
$$
\begin{array}{ll}
\omega(x,0,t) = \Frac {\partial^2 \psi}{\partial y^2}(x,0,t) &
\mbox{ on } \Gamma_1,\\
\omega(x,1,t) = \Frac {\partial^2 \psi}{\partial y^2}(x,1,t) &
\mbox{ on } \Gamma_3,\\
\omega(0,y,t) = \Frac {\partial^2 \psi}{\partial x^2}(0,y,t) &
\mbox{ on } \Gamma_2,\\
\omega(1,y,t) = \Frac {\partial^2 \psi}{\partial x^2}(1,y,t) &
\mbox{ on } \Gamma_4.
\end{array}
$$
So, since $\psi_{\partial \Omega}=0$ and $u=\Frac {\partial
\psi}{\partial y}$, $v=-\Frac {\partial \psi}{\partial x}$ , we
obtain by using Taylor expansions
\begin{eqnarray}
\begin{array}{ll}
\omega_{i,0} &=
\Frac { \psi_{i,1}-8\psi_{i,2}}{2h^2}\label{BC1},\\
\omega_{i,N+1} &= \Frac {- \psi_{i,N-1}+8\psi_{i,N} -6h
g(ih)}{2h^2}\label{BC2},\\
\omega_{0,j} &=
\Frac {\psi_{1,j}-8\psi_{2,j}}{2h^2}i\label{BC3},\\
\omega_{N+1,j} &= \Frac {-
\psi_{N-1,j}+8\psi_{N,j}}{2h^2}\label{BC4}.
\end{array}
\end{eqnarray}
Here $g(x)$ denotes the boundary condition function for the
horizontal velocity
at the boundary $\Gamma_3$.\\
The boundary conditions on $\psi$ are homogeneous Dirichlet BC. The
operators are discretized by second order centered schemes on a
uniform mesh composed by $N$ points in each direction of the domain
of step-size
$h=\Frac{1}{N+1}$. The total number of unknowns is then $2 N^2$.\\
The boundary conditions on $\omega$ are iteratively implemented
according to the relations (\ref{BC1}-\ref{BC4}), making the finite
differences scheme second order accurate. Using the following fourth order accurate extrapolation,

\begin{center}
\begin{equation}\label{NS_extrap4}
\left\lbrace  \begin{array}{rcl}
\omega_{i,0} & = & \dfrac{1}{h^2} \left( 8 \psi_{i,1} - 3 \psi_{i,2} + \dfrac{8}{9} \psi_{i,3} - \dfrac{1}{8} \psi_{i,4} \right),\\
\omega_{i,N+1} & = & \dfrac{1}{h^2} \left( 8 \psi_{i,N} - 3 \psi_{i,N-1} + \dfrac{8}{9} \psi_{i,N-2} - \dfrac{1}{8} \psi_{i,N-3} \right) - \dfrac{25}{6h} g(ih), \\
\omega_{0,j} & = & \dfrac{1}{h^2} \left( 8 \psi_{1,j} - 3 \psi_{2,j} + \dfrac{8}{9} \psi_{3,j} - \dfrac{1}{8} \psi_{4, j} \right),\\
\omega_{N+1, j} & = & \dfrac{1}{h^2} \left( 8 \psi_{N,j} - 3 \psi_{N-1,j} + \dfrac{8}{9} \psi_{N-2,j} - \dfrac{1}{8} \psi_{N-3,j} \right),
\end{array}  \right. 
\end{equation}
\end{center}
we complete the discretization. Now the implementation of the RSS scheme reads as

\begin{itemize}
\item Convection-Diffusion problem: knowing $\psi^{(k)}$, compute $\omega^{(k+1)}$ solution of
\begin{center}
\begin{equation}
\left\lbrace  \begin{array}{rcll}
\dfrac{\omega^{(k+1)}-\omega^{(k)}}{\Delta t}-\dfrac{1}{Re} \Delta \omega^{(k+1)} + \dfrac{\partial \psi^k}{\partial y} \dfrac{\partial \omega^{(k)}}{\partial x}  -  \dfrac{\partial \psi^{(k)}}{\partial x} \dfrac{\partial \omega^{(k)}}{\partial y}& = & 0 & \text{ in } \Omega = ]0, 1[^2 \\
\omega^{(k+1)} & = & \Delta \psi^{(k)} & \text{ on } \partial \Omega \\
\end{array}  \right. 
\label{CD_NS}
\end{equation}
\end{center}
\item  Poisson problem: knowing $\omega^{(k+1)}$, compute $\psi^{(k+1)}$ solution of
\begin{center}
\begin{equation}
\left\lbrace  \begin{array}{rcll}
\omega^{(k+1)} & = & \Delta \psi^{(k+1)} & \text{ in } \Omega = ]0, 1[^2, \\
\psi^{(k+1)} & = & 0 & \text{ on} \partial \Omega .\\
\end{array}  \right. 
\label{Poisson_NS}
\end{equation}
\end{center}
\end{itemize}

\begin{center}
\begin{minipage}[H]{12cm}
  \begin{algorithm}[H]
    \caption{: RRS Navier-Stokes}\label{NS}
    \begin{algorithmic}[1]
        \State  $(\omega^0, \psi^0)$ given as solution of the Stokes problem (\ref{Stokes_psi})
            \For{$k=0,1, \cdots$until convergence}
             \State Update  the boundary terms using  (\ref{NS_extrap4})
              \State Compute  $\omega^{(k+1)}$ by solving. (\ref{CD_NS}) with RSS (\ref{NLRSS})
               \State Compute $\psi^{n+1}$as solution of the  Poisson equation (\ref{Poisson_NS})      
            \EndFor
    \end{algorithmic}
    \end{algorithm}
\end{minipage}
\end{center}

\subsection{RSS Schemes for computing Steady States of the lid driven cavity}

We give now numerical results on the square cavity $\Omega=]0,1[^2$, we compare the numerical values of the steady state with those of the literature. Here $g(x)=1$.

The steady state is computed for $\| \frac{\partial \psi }{\partial t}  \| < \varepsilon = 10^{-3}$. We report hereafter the vorticity and the stream function in figures \ref{NS_Re100}, \ref{NS_Re400}, \ref{NS_Re1000} and \ref{NS_Re3200}, for $Re=100, Re=400, Re=1000, Re=3200$ respectively. They agree with those of the literature \cite{BenArtziCroisille,BruneauJouron,Ghia,Goyon}; the localization of the extrema of $\omega$ and $\psi$ are reported on table \ref{Vortex}.

\begin{figure}[!h]
\begin{center}
\includegraphics[width=5.5cm, height=5.5cm]{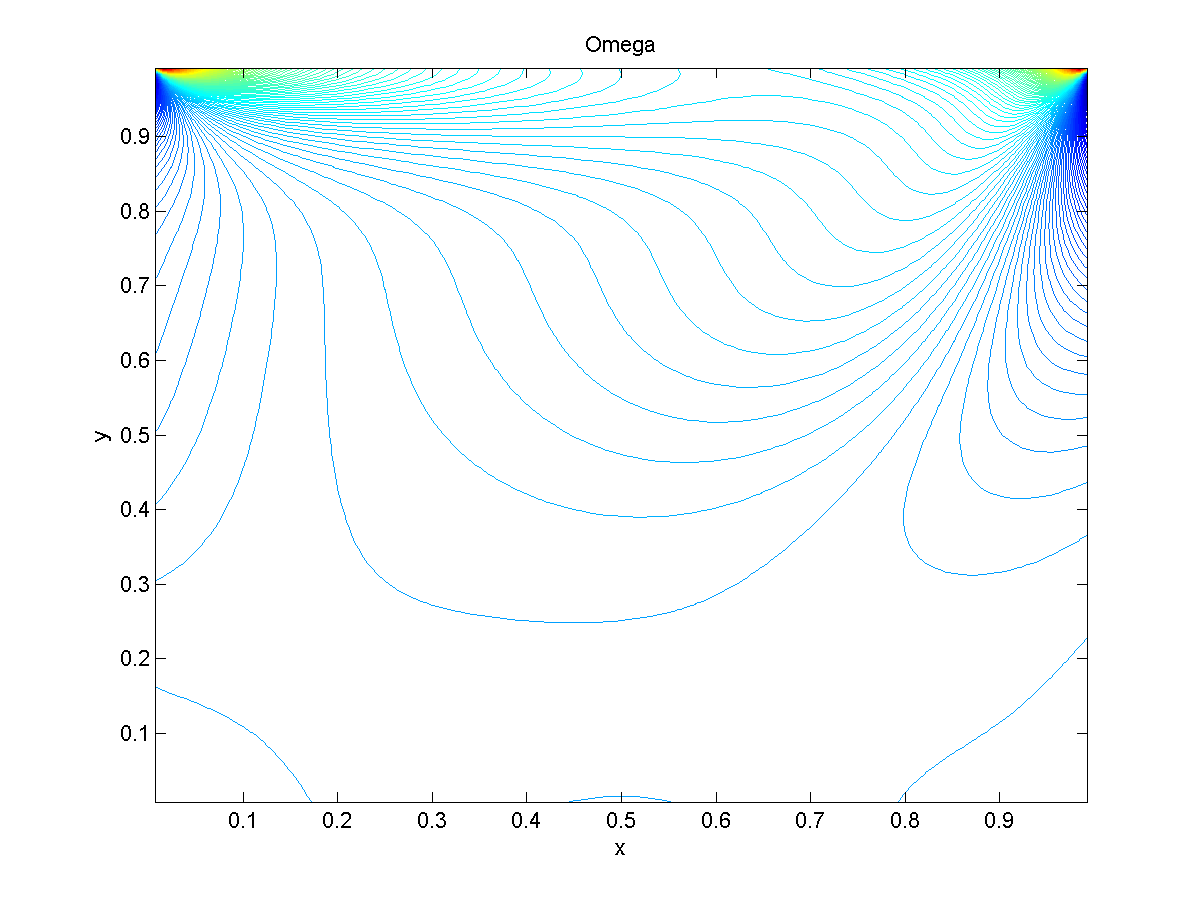}
\includegraphics[width=5.5cm, height=5.5cm]{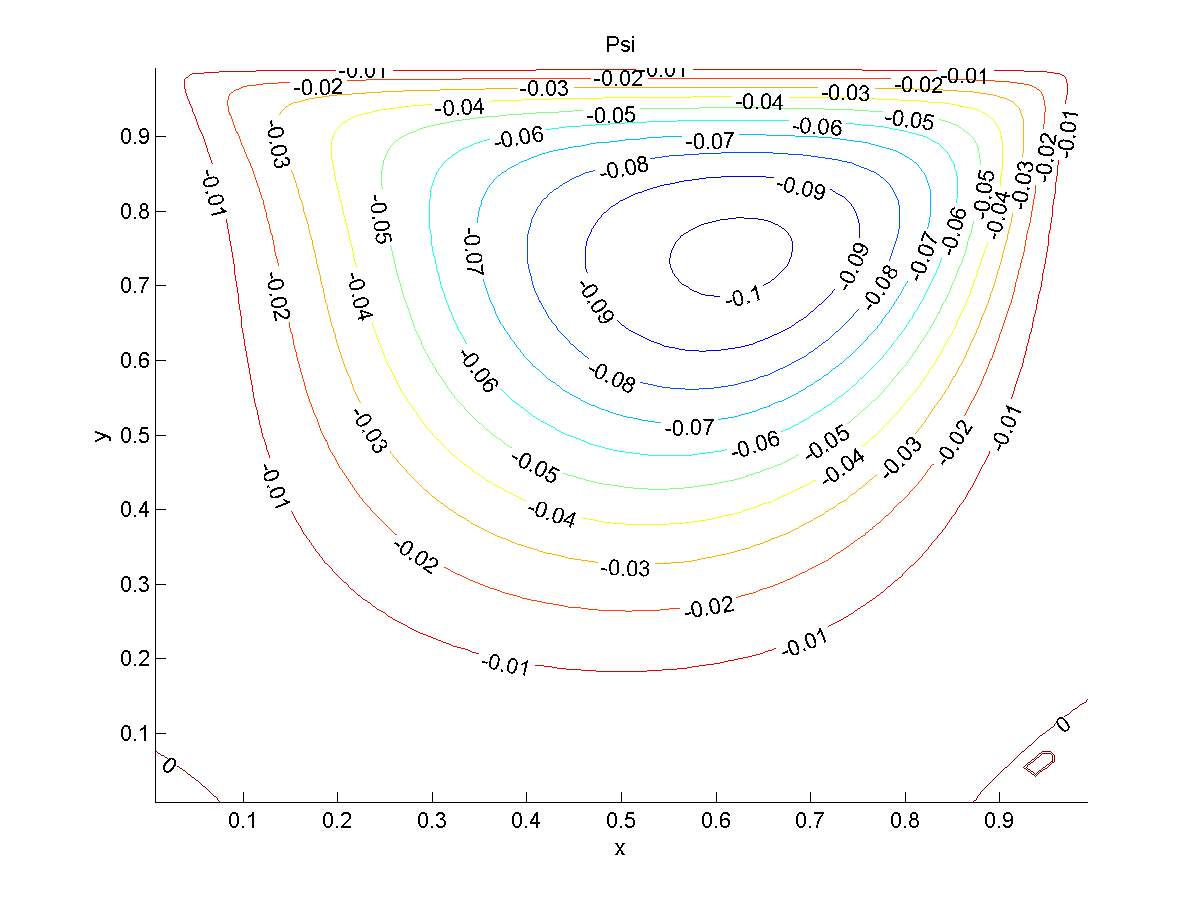}\\
\caption{Steady solution of NSE (\ref{Navier_Stokes_psi}) - $g \equiv 1$ - $\tau = 1$ -  $N=127$ - $Re = 100$ - $\Delta t = 0.001$}
\label{NS_Re100}
\end{center}
\end{figure}

\begin{figure}[!h]
\begin{center}
\includegraphics[width=5.5cm, height=5.5cm]{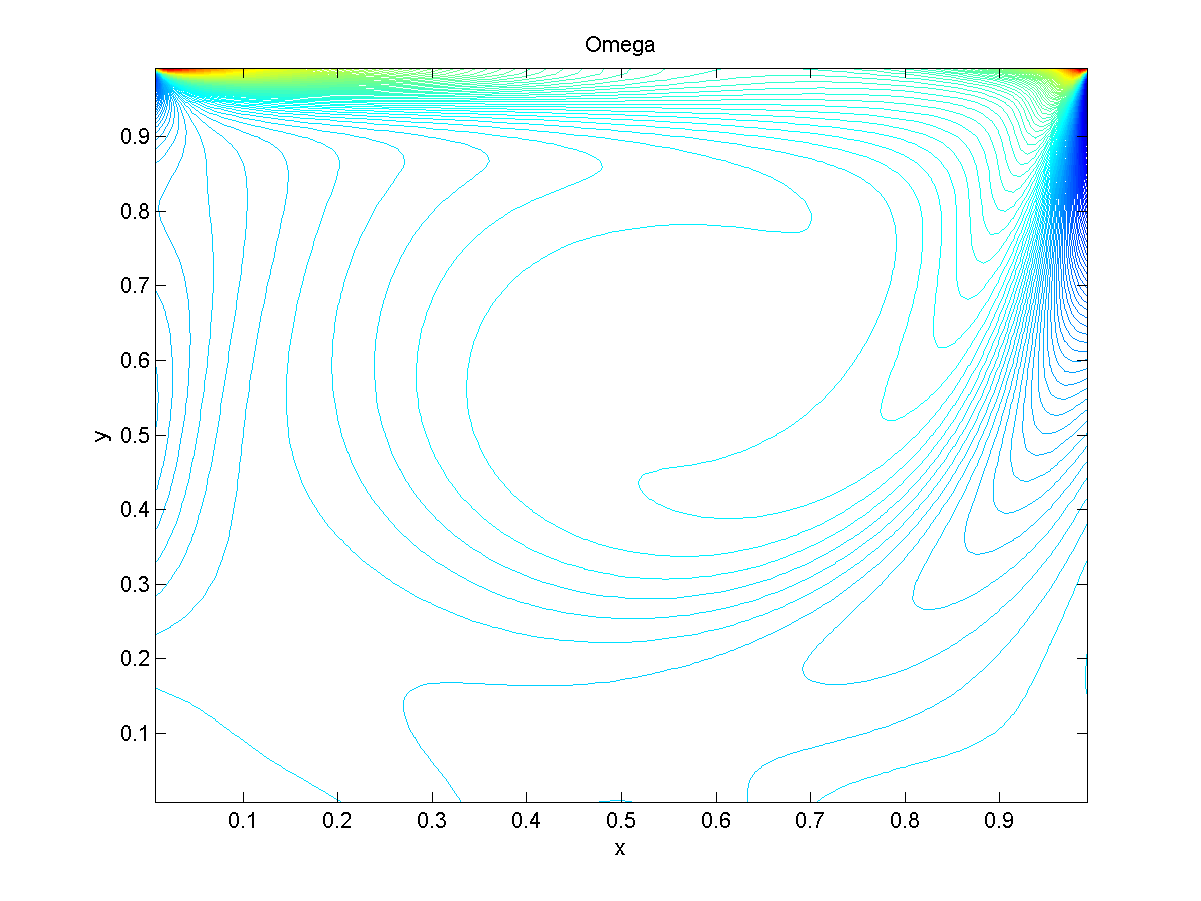}
\includegraphics[width=5.5cm, height=5.5cm]{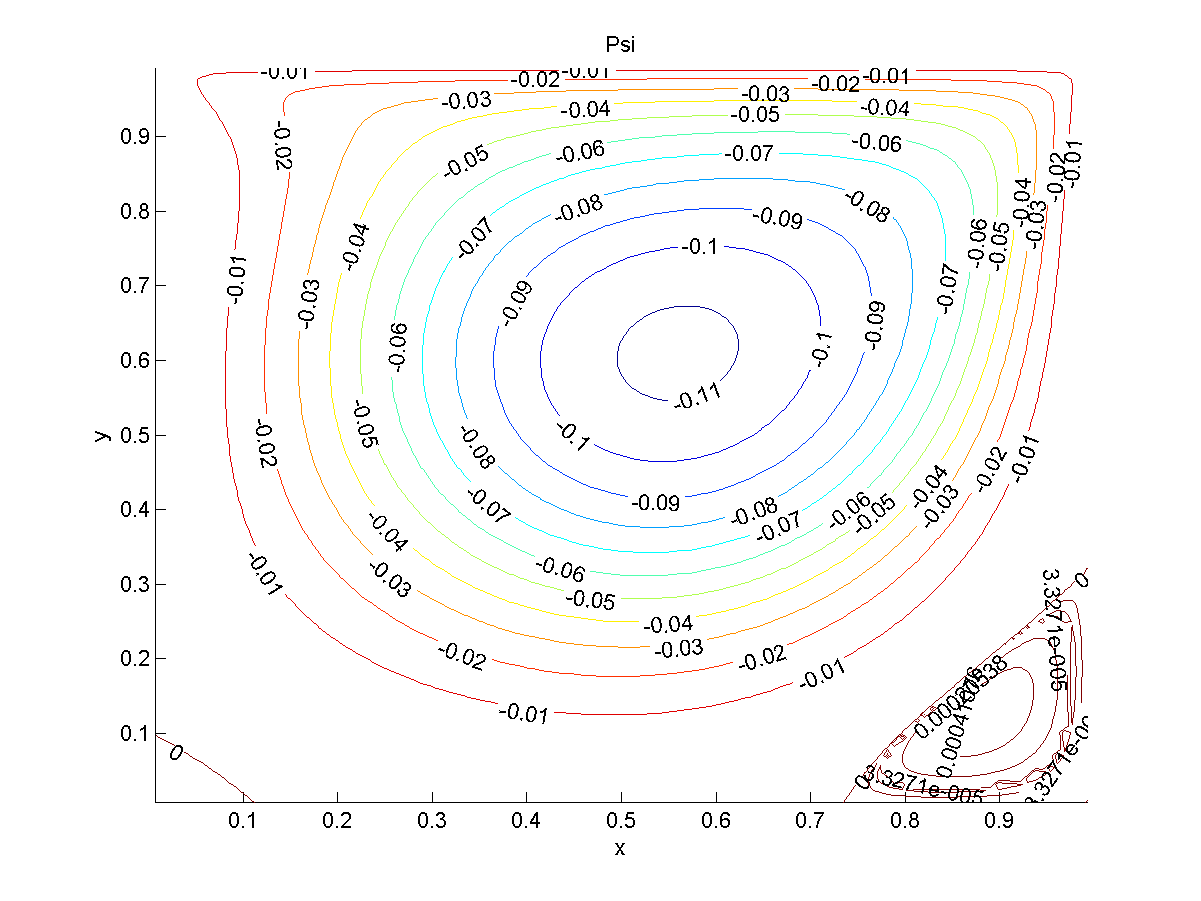}\\
\caption{Solution of NSE  (\ref{Navier_Stokes_psi}) - $g \equiv 1$ - $\tau = 1$ -  $N=127$ - $Re = 400$ - $\Delta t = 0.001$}
\label{NS_Re400}
\end{center}
\end{figure}

\begin{figure}[!h]
\begin{center}
\includegraphics[width=5.5cm, height=5.5cm]{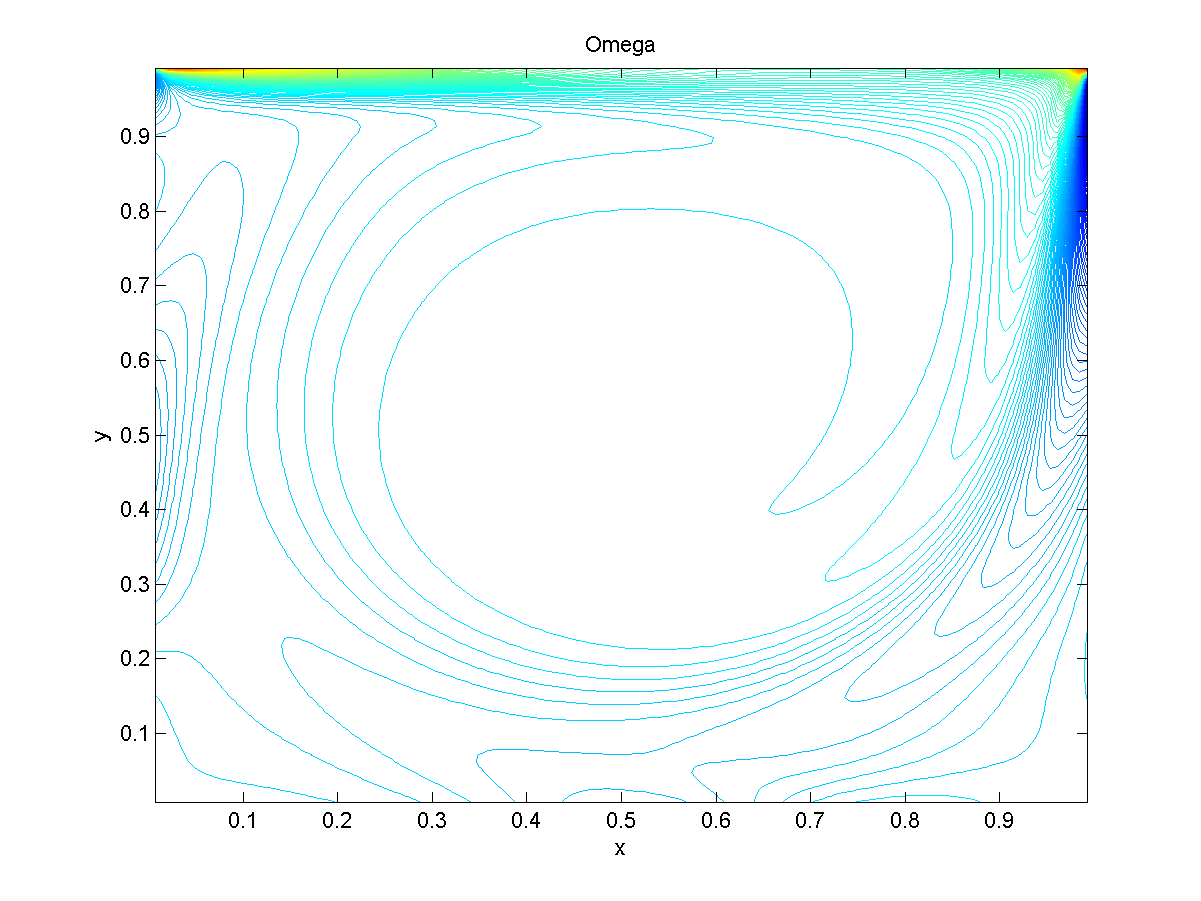}
\includegraphics[width=5.5cm, height=5.5cm]{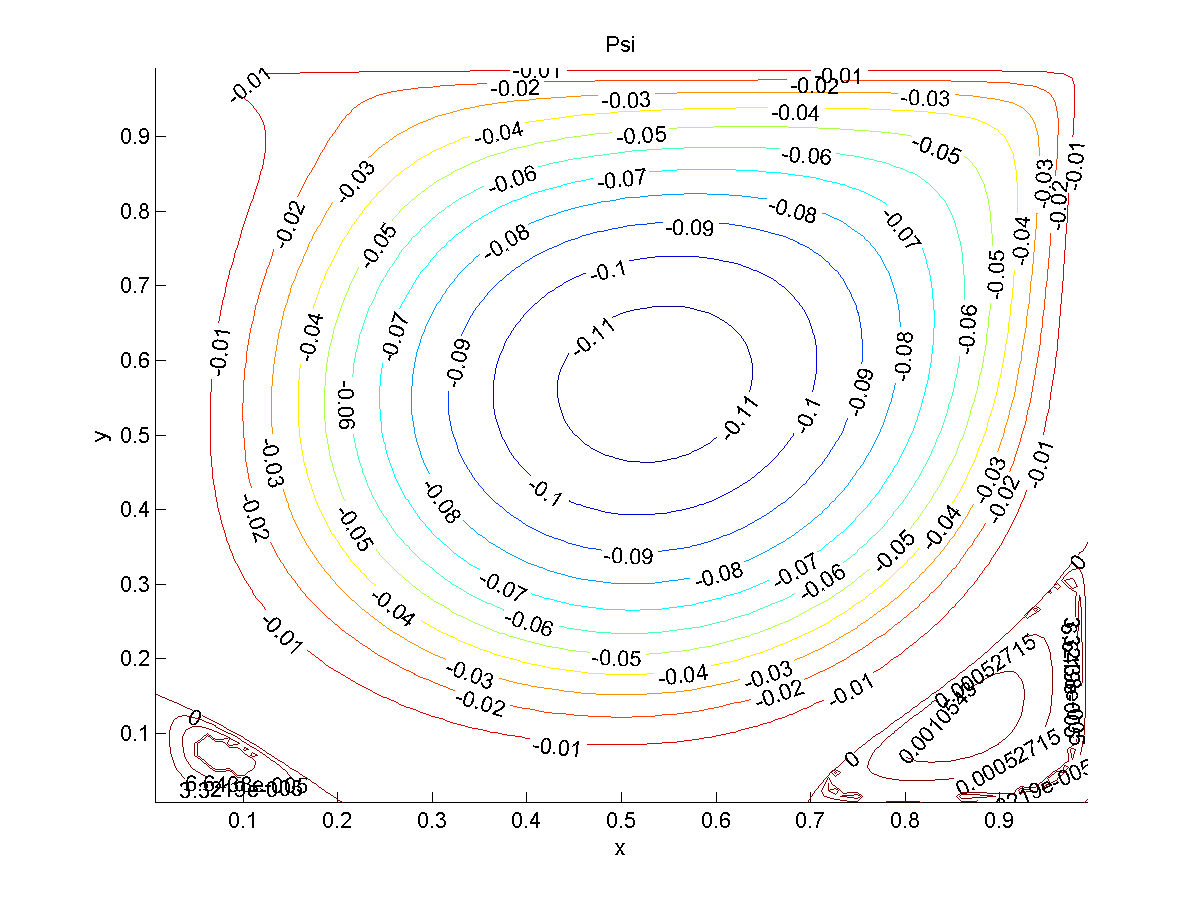}\\
\caption{Solution of NSE (\ref{Navier_Stokes_psi}) - $g \equiv 1$ - $\tau = 1$ -  $N=127$ - $Re = 1000$ - $\Delta t = 0.0005$}
\label{NS_Re1000}
\end{center}
\end{figure}

\begin{figure}[!h]
\begin{center}
\includegraphics[width=5.5cm, height=5.5cm]{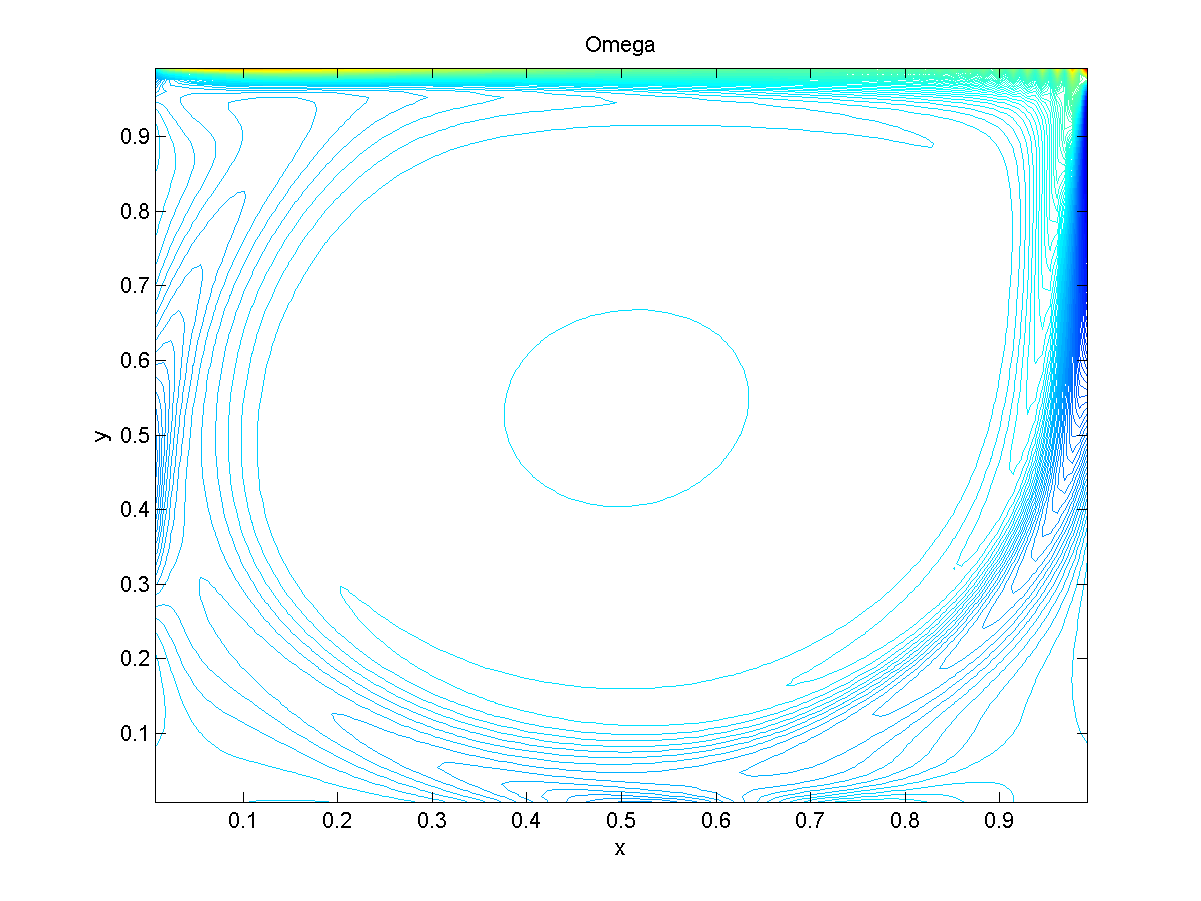}
\includegraphics[width=5.5cm, height=5.5cm]{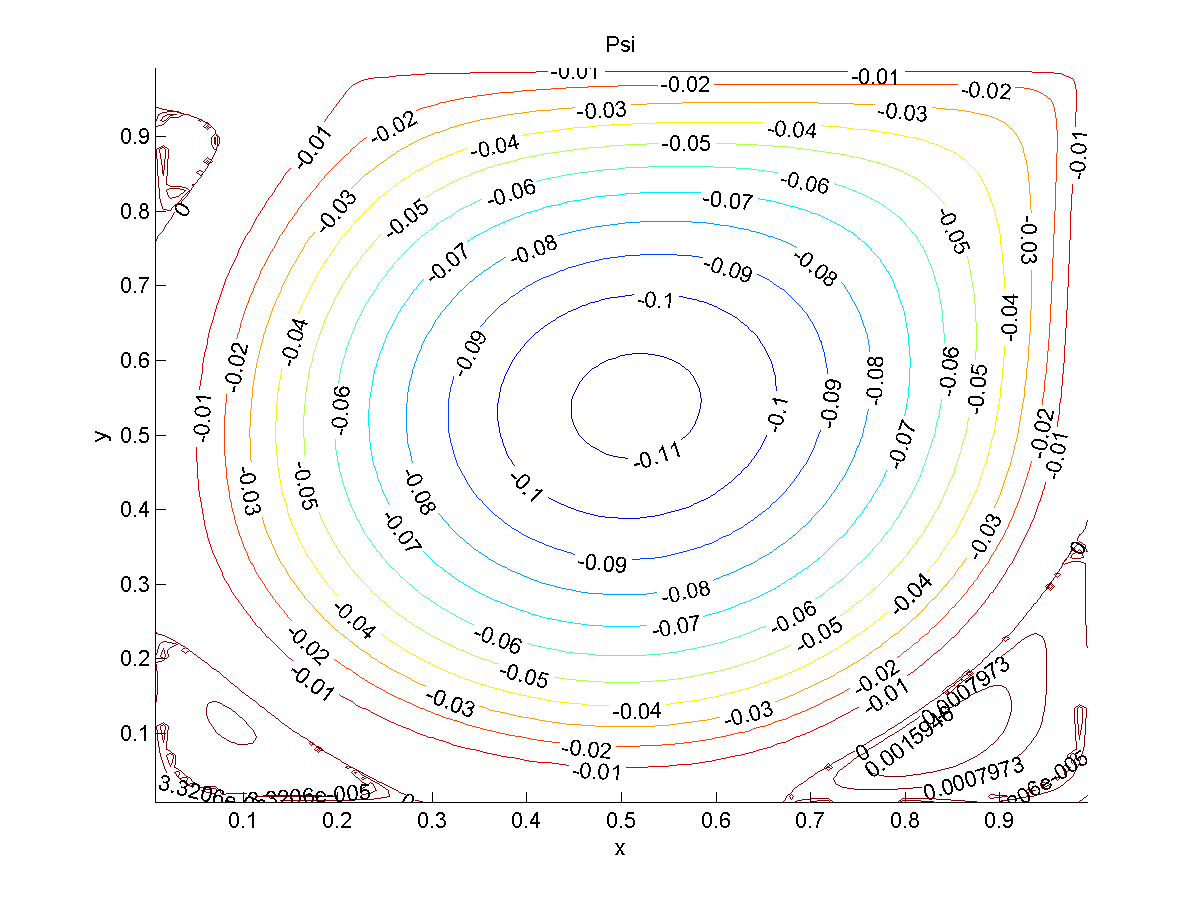}\\
\caption{Solution of NSE (\ref{Navier_Stokes_psi}) - $g \equiv 1$ - $\tau = 1$ -  $N=127$ - $Re = 3200$ - $\Delta t = 0.0005$}
\label{NS_Re3200}
\end{center}
\end{figure}
\begin{table}[!h]
\begin{center}

\vspace{0.5cm}

$Re=100$

\begin{tabular}{| c|c|c|c|c|c | }
\hline 
 & Principal Vortex  & Ben Artzi {\it et al} \cite{BenArtziCroisille} & O. Goyon \cite{Goyon} & extrapoled RSS \\ 
Spatial accuracy &                  & 4th order compact scheme                        & second order      &   4th order          \\
\hline 
grid &   & $97 \times 97$ & $ 129 \times 129$ & $127 \times 127$ \\ 
\hline
$\Delta t$ & & &$0.005$ & $0.004$\\
\hline
&  intensity &     & 0.1033 & 0.1026 \\ 
 
  & x &   & 0.6172 & 0.6172 \\ 

  & y &   & 0.7343 & 0.7422 \\ 
\hline 
\end{tabular}

\vspace{0.5cm}

$Re=400$

\begin{tabular}{| c|c|c|c|c|c | }
\hline 
 & Principal Vortex  & Ben Artzi {\it et al} \cite{BenArtziCroisille} & O. Goyon \cite{Goyon} & extrapoled RSS \\ 
Spatial accuracy &                  & 4th order compact scheme                        & second order      &   4th order          \\
\hline 
grid &   & $97 \times 97$ & $ 129 \times 129$ & $127 \times 127$ \\ 
\hline
 $\Delta t$&  & & &$0.017$\\
 \hline
& intensity & 0.1136 &   & 0.1123 \\ 

  & x & 0.5521 &   & 0.5625 \\ 

  & y & 0.6042 &   & 0.6094 \\ 
\hline 
\end{tabular}

\vspace{0.5cm}

$Re=1000$

\begin{tabular}{| c|c|c|c|c|c | }
\hline 
 & Principal Vortex  & Ben Artzi {\it et al} \cite{BenArtziCroisille} & O. Goyon \cite{Goyon} & extrapoled RSS \\ 
Spatial accuracy &                  & 4th order compact scheme                        & second order      &   4th order          \\
\hline 
grid &   & $97 \times 97$ & $ 129 \times 129$ & $127 \times 127$ \\ 
\hline
$\Delta t$ & & &$0.02$ & $0.01$\\ 
\hline
&  intensity  & 0.1178 & 0.1157 & 0.1158 \\ 

  & x & 0.5312 & 0.5312 & 0.5391 \\ 
 
  & y & 0.5625 & 0.5625 & 0.5703 \\ 
\hline 
\end{tabular}

\vspace{0.5cm}

$Re=3200$

\begin{tabular}{| c|c|c|c|c|c | }
\hline 
 & Principal Vortex  & Ben Artzi {\it et al} \cite{BenArtziCroisille} & O. Goyon \cite{Goyon} & extrapoled RSS \\ 
Spatial accuracy &                  & 4th order compact scheme                        & second order      &   4th order          \\
\hline 
grid &   & $97 \times 97$ & $ 129 \times 129$ & $127 \times 127$ \\ 
\hline
$\Delta t$ & & &$0.01$ & $0.006$\\ 
\hline
&intensity  & 0.1174 & 0.1122 & 0.1129 \\ 

  & x & 0.5208 & 0.5234 & 0.5156 \\ 

  & y & 0.5417 & 0.5468 & 0.5391 \\ 
\hline 
\end{tabular} 
\caption{Extremal values of  $\psi$ and their localization for NSE (\ref{Navier_Stokes_psi}). Comparison with results of Croisille and Goyon, for different Reynolds numbers. $g \equiv 1$ - $\tau = 1$}
\label{Vortex}
\end{center}
\end{table}

Now we illustrate the influence of the stabilization parameter $\tau$ on the convergence in time to the steady state.
A large value of $\tau$ allows to take a large time step $\Delta t$ but slows down the convergence in time.
We have plotted in figures  (\ref{NS_dudt_Re100_tau1_N63}) and (\ref{NS_dudt_Re100_tau100_N63})  the evolution in time of $\| \frac{\partial \psi}{\partial t} \|$. We observe that, for $\tau=1$ the RSS schemes (first and second order) behave similarly as the reference one (semi backward Euler's); for $\tau =100$, we see that RSS is slow downed while its extrapoled version has comparable dynamics to Euler's.

\begin{figure}[!h]
\begin{center}
\includegraphics[width=10cm]{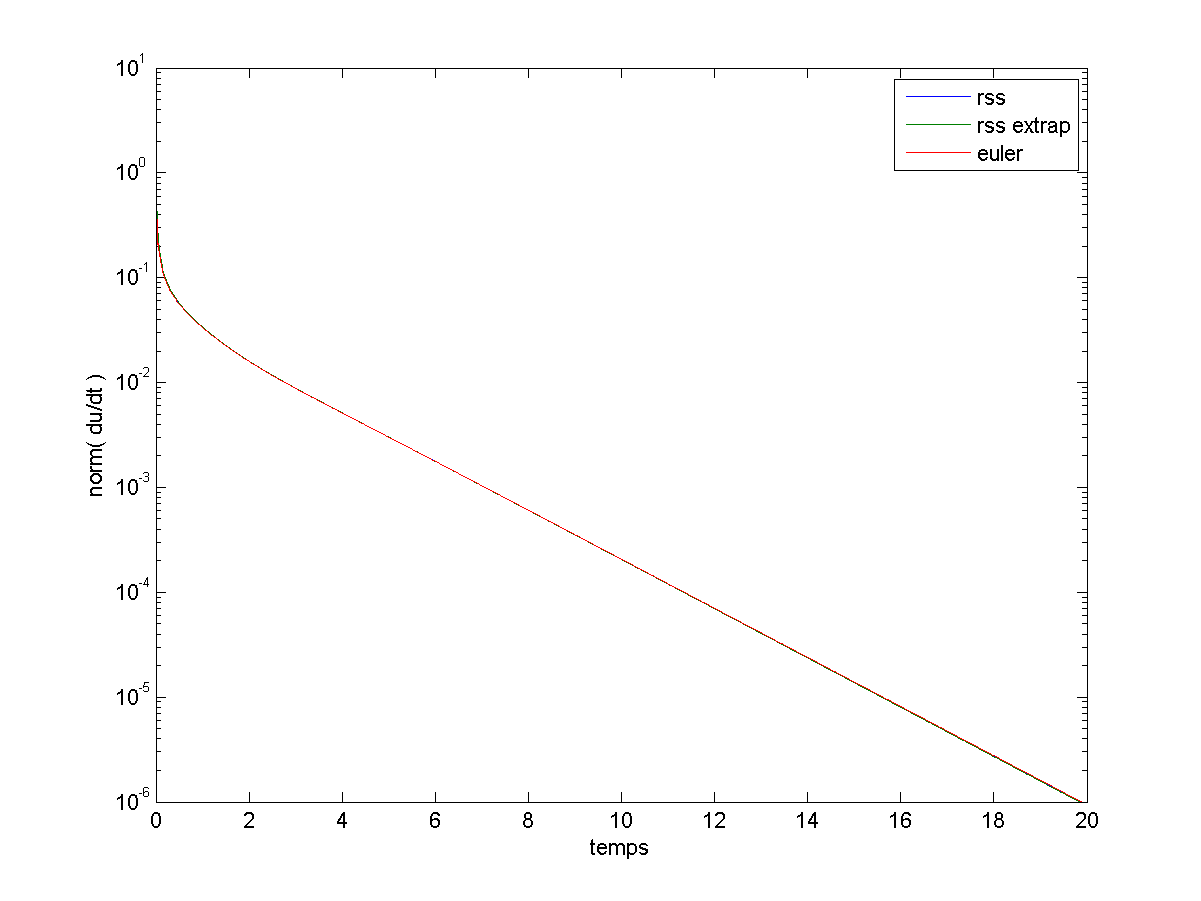}
\caption{Convergence to NSE steady state (\ref{Navier_Stokes_psi}) - $\tau = 1$ -  $N=63$ - $Re = 100$ - $\Delta t = 0.01$}
\label{NS_dudt_Re100_tau1_N63}
\end{center}
\end{figure}

\begin{figure}[!h]
\begin{center}
\includegraphics[width=10cm]{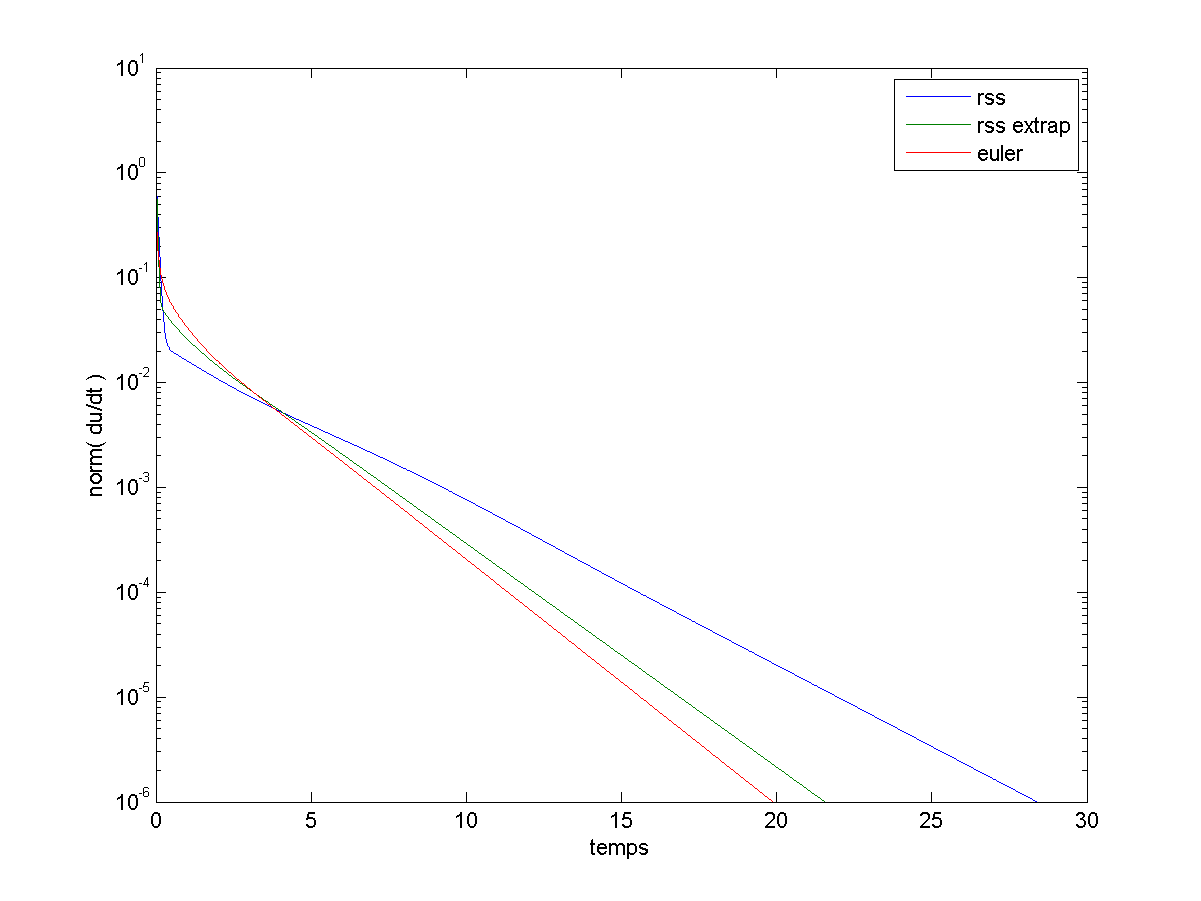}
\caption{Convergence to NSE steady state (\ref{Navier_Stokes_psi}) - $\tau = 100$ -  $N=63$ - $Re = 100$ - $\Delta t = 0.01$}
\label{NS_dudt_Re100_tau100_N63}
\end{center}
\end{figure}

We now give numerical results for the rectangular cavity. They are presented in figures  (\ref{NSrect_Re100}) ,(\ref{NSrect_Re400}), (\ref{NSrect_Re1000}) and (\ref{NSrectRe3200}) for $Re=100, \ Re=400$, $Re=1000$, $Re=3200$ respectiveley. They agree with those obtained by Goyon \cite{Goyon}, see also the numerical values reported in Table \ref{Vortex_rect}.

\begin{figure}[!h]
\begin{center}
\includegraphics[width=4.5cm, height=8cm]{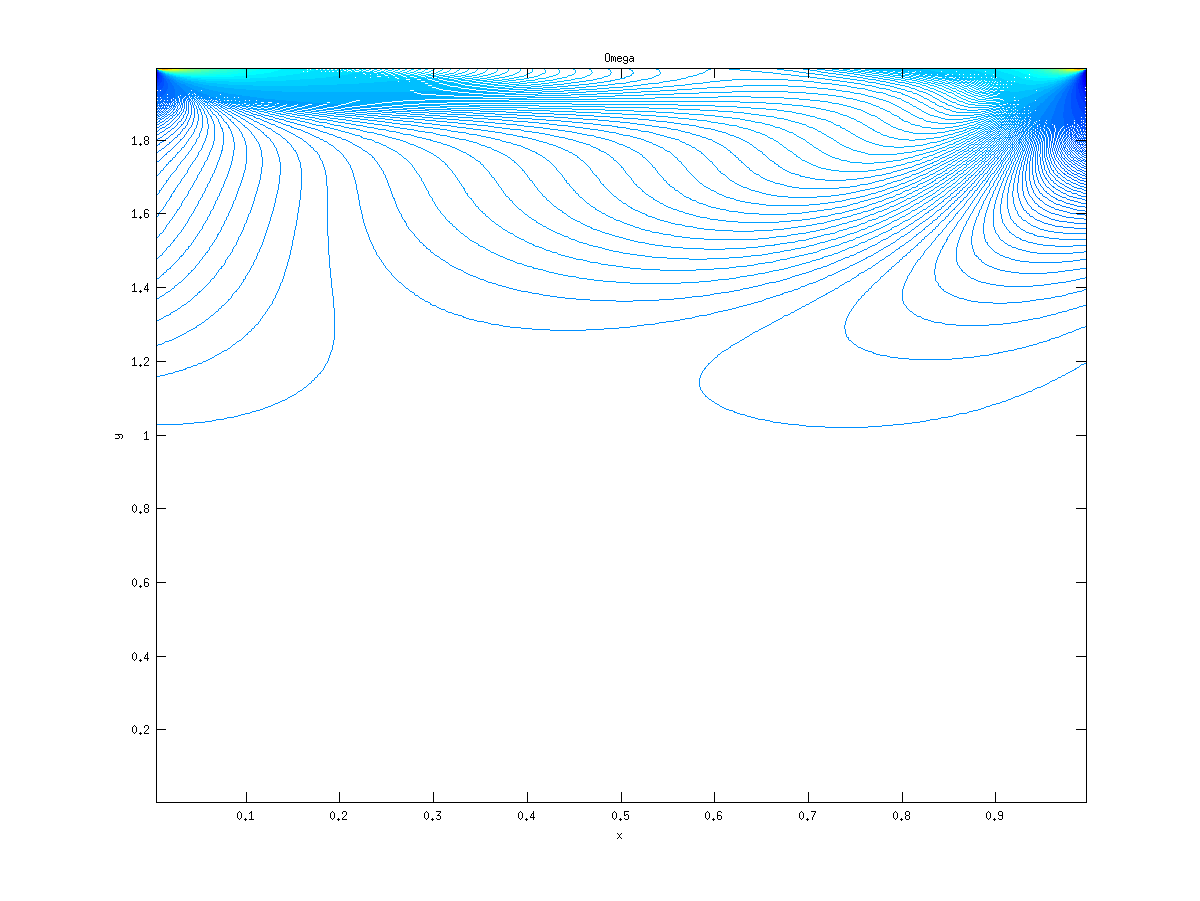}
\includegraphics[width=4.5cm, height=8cm]{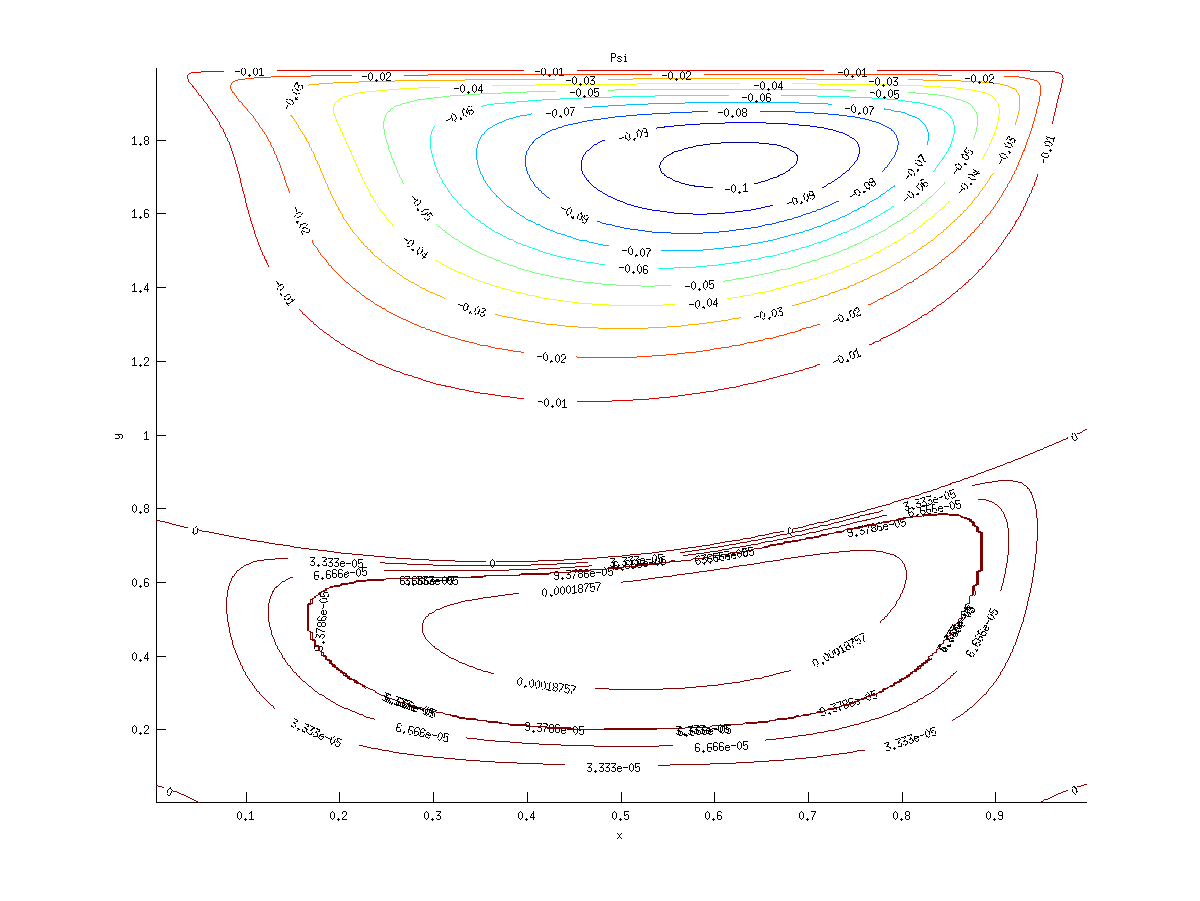}\\
\caption{Solution of (\ref{Navier_Stokes_psi}) in $[0; 1] \times [0; 2]$ - $g \equiv 1$ - $\tau = 1$ -  $255 \times 511$ - $Re = 100$ - $\Delta t = 0.001$}
\label{NSrect_Re100}
\end{center}
\end{figure}

\begin{figure}[!h]
\begin{center}
\includegraphics[width=4.5cm, height=8cm]{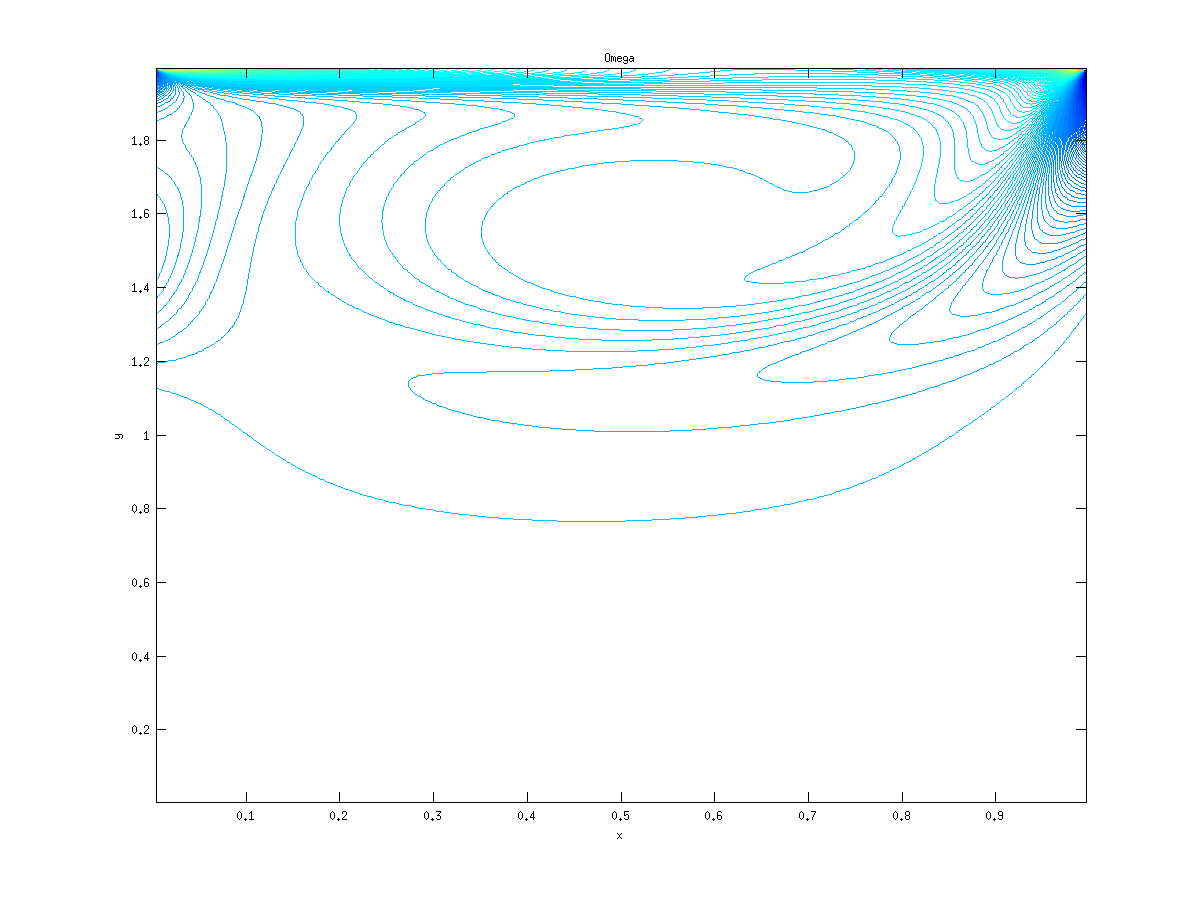}
\includegraphics[width=4.5cm, height=8cm]{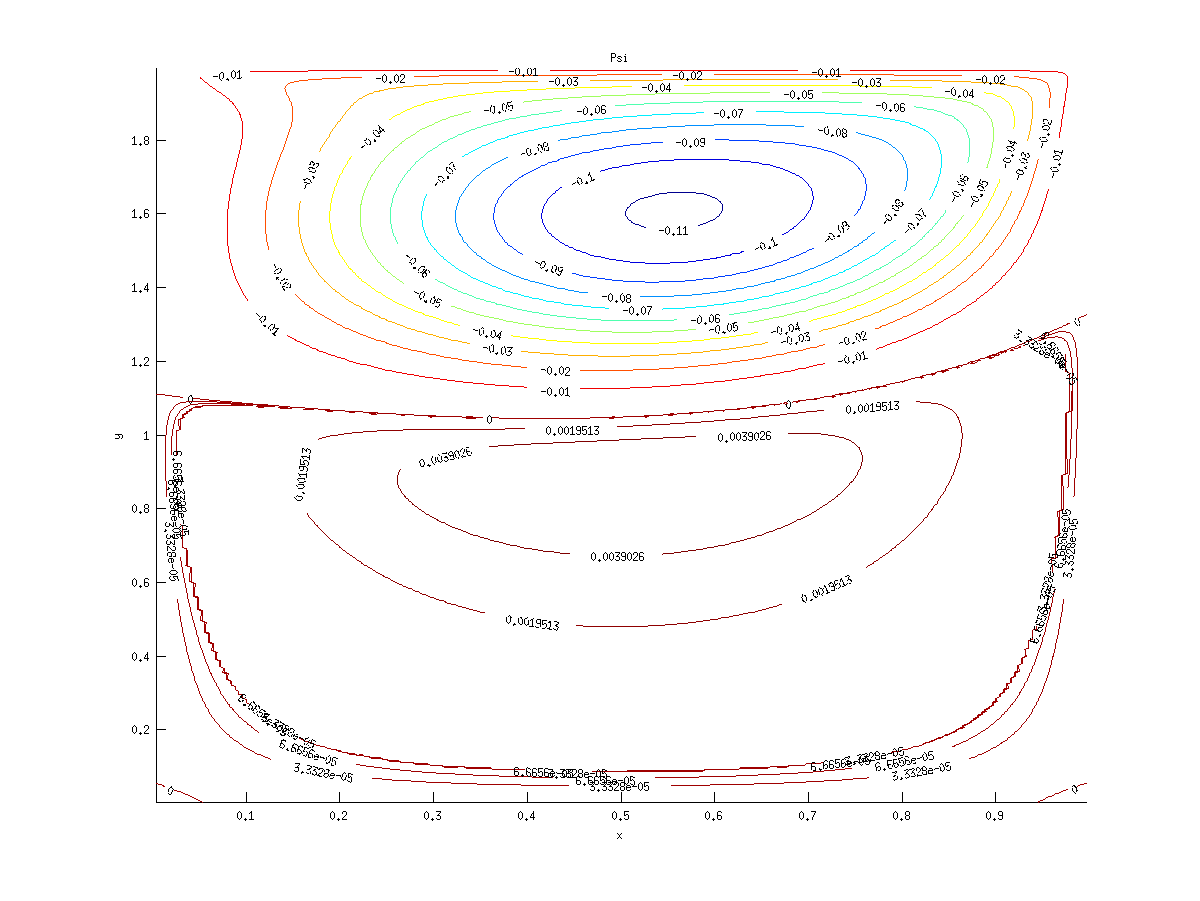}\\
\caption{Solution of NSE (\ref{Navier_Stokes_psi}) in $[0; 1] \times [0; 2]$ - $g \equiv 1$ - $\tau = 1$ -  $255 \times 511$ - $Re = 400$ - $\Delta t = 0.001$}
\label{NSrect_Re400}
\end{center}
\end{figure}

\begin{figure}[!h]
\begin{center}
\includegraphics[width=4.5cm, height=8cm]{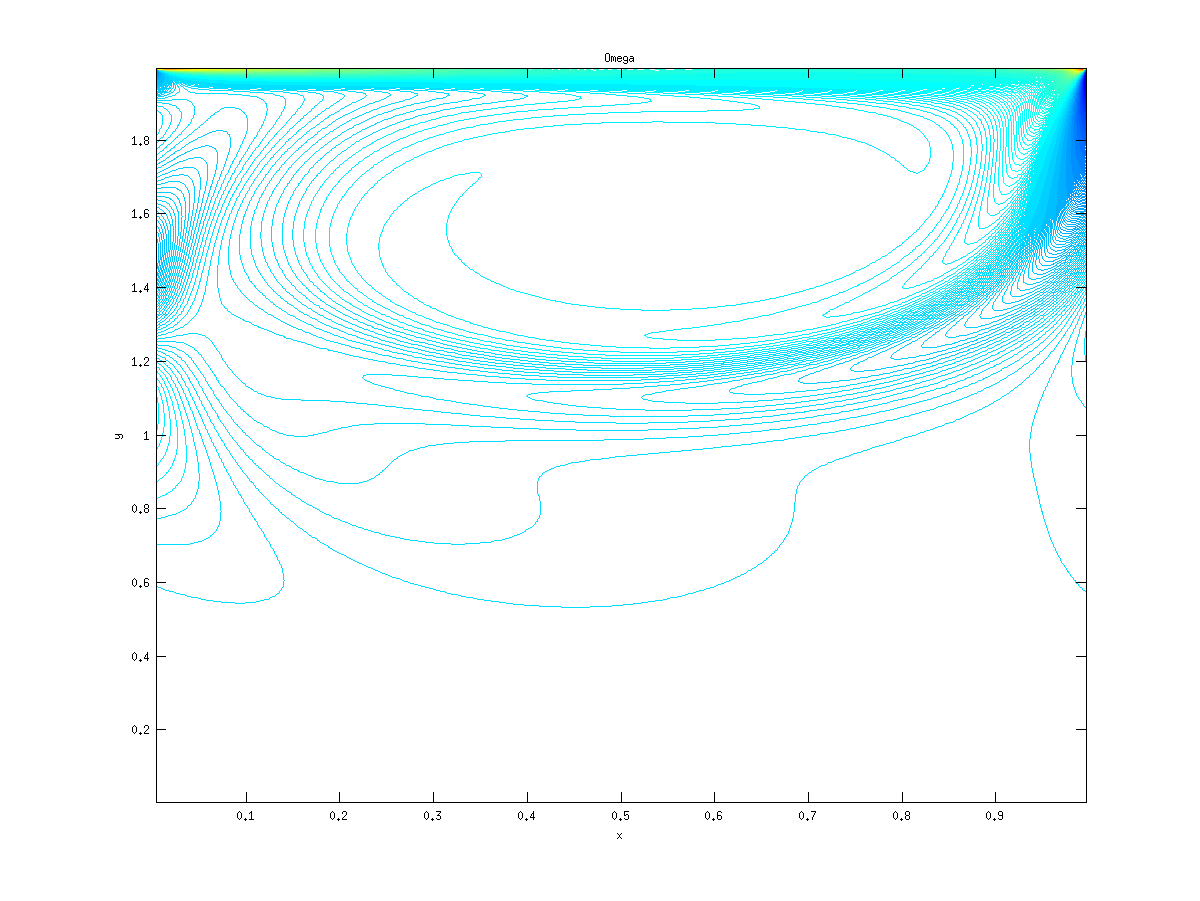}
\includegraphics[width=4.5cm, height=8cm]{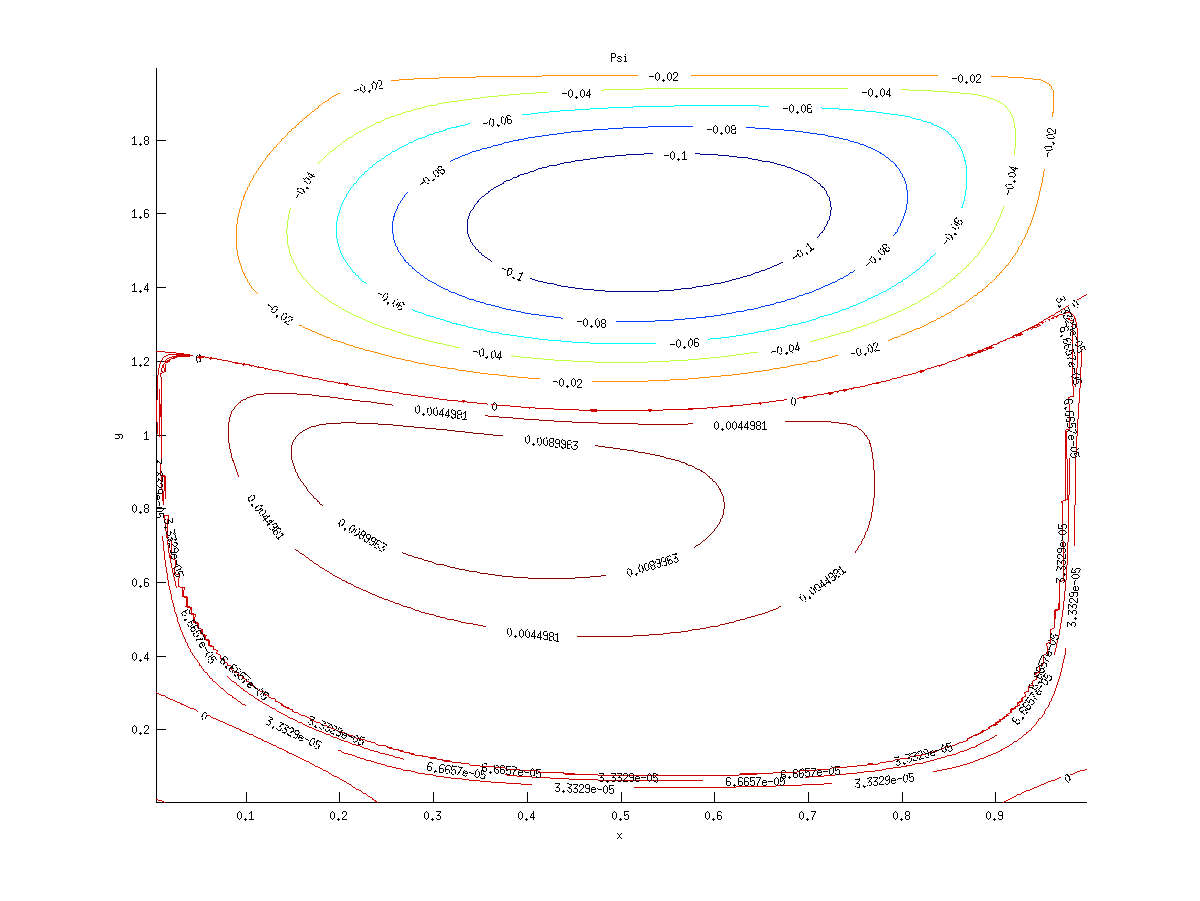}\\
\caption{Solution of NSE (\ref{Navier_Stokes_psi}) in $[0; 1] \times [0; 2]$ - $g \equiv 1$ - $\tau = 1$ -  $255 \times 511$ - $Re = 1000$ - $\Delta t = 0.0005$}
\label{NSrect_Re1000}
\end{center}
\end{figure}

\begin{table}[!h]
\begin{center}

$Re=100$

\begin{tabular}{| c|c|c|c|c|c | }
\hline 
                &             & extrapoled RSS     &  extrapoled RSS    & O. Goyon \cite{Goyon} & C-H. Bruneau and\\ 

                &             &                  &                  & second order               & C. Jouron \cite{BruneauJouron}  \\ 
\hline 
grid:      &             & $63 \times 127$  & $255 \times 511$ & $65 \times 129$       & multi-grid \\ 
\hline 
$\Delta t$ &        &       $1.10^{-3}$                &            $1.10^{-3}$            &        $1.10^{-2}$                           & \\
\hline 
$\epsilon$:         &             & $10^{-5}$        & $10^{-3}$        &         $10^{-5}$              &\\ 
\hline 
VS              & $\psi$      & 0.1040           & 0.1034           & 0.1035                & 0.1033\\ 

                & x           & 0.6094           & 0.6172           & 0.6093                & 0.6172\\ 

                & y           & 1.7344           & 1.7344           & 1.7343                & 1.7344\\ 
\hline 
VI              & $\psi$      & $8 \times 10^{-3}$& 0.0003           & $6.65 \times 10^{-4}$& $0.783 \times 10^{-3}$\\ 

                & x           & 0.5469          & 0.5820           & 0.5468                 & 0.5391\\ 
                
                & y           & 0.5938           & 0.5039           & 0.5781                & 0.5859\\ 
\hline
\end{tabular} 

\vspace{0.5cm}

$Re=400$

\begin{tabular}{ |  c|c|c|c|c|c | }
\hline 
                &             &  extrapoled RSS    &  extrapoled RSS    & O. Goyon \cite{Goyon} & C-H. Bruneau and\\ 

                &             &                  &                  & second order              & C. Jouron \cite{BruneauJouron}\\ 
\hline 
grid:      &             & $63 \times 127$  & $255 \times 511$ & $65 \times 129$       & multi-grid \\ 
\hline 
$\Delta t$ &        &       $1.10^{-3}$                &            $1.10^{-3}$            &        $1.10^{-2}$                           & \\
\hline 
$\epsilon$:         &             & $10^{-5}$        & $10^{-3}$        &        $10^{-5}$               & \\ 
\hline 
VS              & $\psi$      & 0.1131           & 0.1120           & 0.1097                & 0.1124\\ 

                & x           & 0.5469           & 0.5586           & 0.5625                & 0.5547\\ 

                & y           & 1.6094           & 1.6094           & 1.6094                & 1.5938\\ 
\hline 
VI              & $\psi$      & 0.009            & 0.0059           & $8.06 \times 10^{-3}$ & $0.909 \times 10^{-2}$\\ 

                & x           & 0.4219           & 0.5156           & 0.4375                & 0.4297\\ 
                
                & y           & 0.8438           & 0.8750           & 0.8593                & 0.8125\\ 
\hline
\end{tabular} 

\vspace{0.5cm}

$Re=1000$

\begin{tabular}{|  c|c|c|c|c|c | }
\hline 
                &             &  extrapoled RSS    &  extrapoled RSS    & O. Goyon \cite{Goyon} & C-H. Bruneau and \\  

                &             &                  &                  & second order               & C. Jouron \cite{BruneauJouron}\\ 
\hline 
grid:      &             & $127 \times 257$ & $255 \times 511$ & $257 \times 513$      & multi-grid \\ 
\hline
$\Delta t$ &        &       $5.10^{-4}$                &            $5.10^{-4}$            &        $5.10^{-3}$                           & \\
\hline 
$\epsilon$:         &             & $10^{-5}$        & $10^{-5}$        &      $10^{-5}$                 & \\ 
\hline 
VS              & $\psi$      & 0.1189           &  0.1196          & 0.1187                & 0.1169\\ 

                & x           & 0.5312           &  0.5312          & 0.5313                & 0.5273\\ 

                & y           & 1.5781           &  1.5781          & 1.5781                & 1.5625\\ 
\hline 
VI              & $\psi$      & 0.0134           &  0.0135          & $1.32 \times 10^{-2}$ & 0.0148\\ 

                & x           & 0.3438           &  0.3398          & 0.3359                & 0.3516\\ 
                
                & y           & 0.8438           &  0.8438          & 0.8476                & 0.7891\\ 
\hline
\end{tabular}

\caption{Some extremal values of $\omega$ and $\psi$ for the steady state of NSE (\ref{Navier_Stokes_psi}) on $[0; 1] \times [0; 2]$ (upper and lower vortex) - $g \equiv 1$ - $\tau = 1$}
\label{Vortex_rect}
\end{center}
\end{table}
\clearpage

\begin{figure}[!h]
\begin{center}
\includegraphics[width=4.5cm, height=8cm]{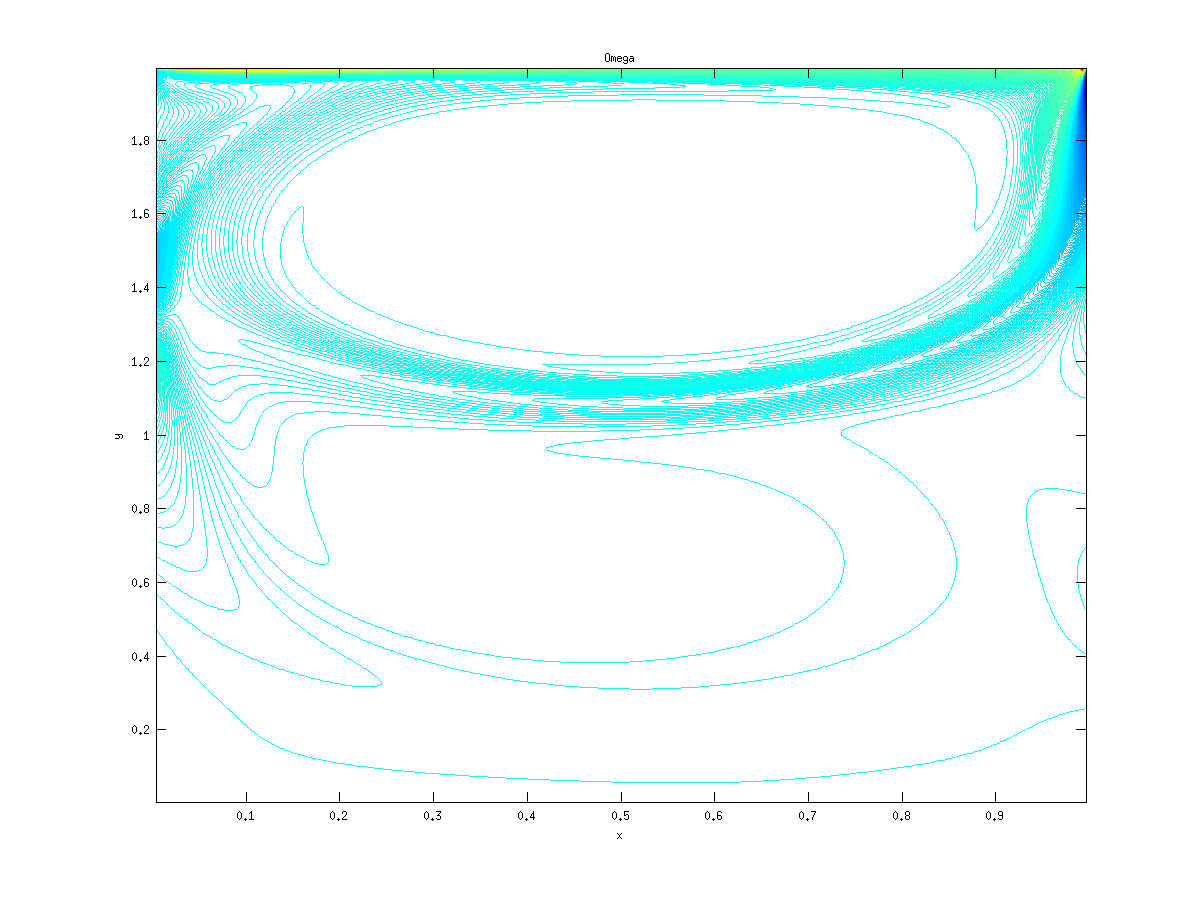}
\includegraphics[width=4.5cm, height=8cm]{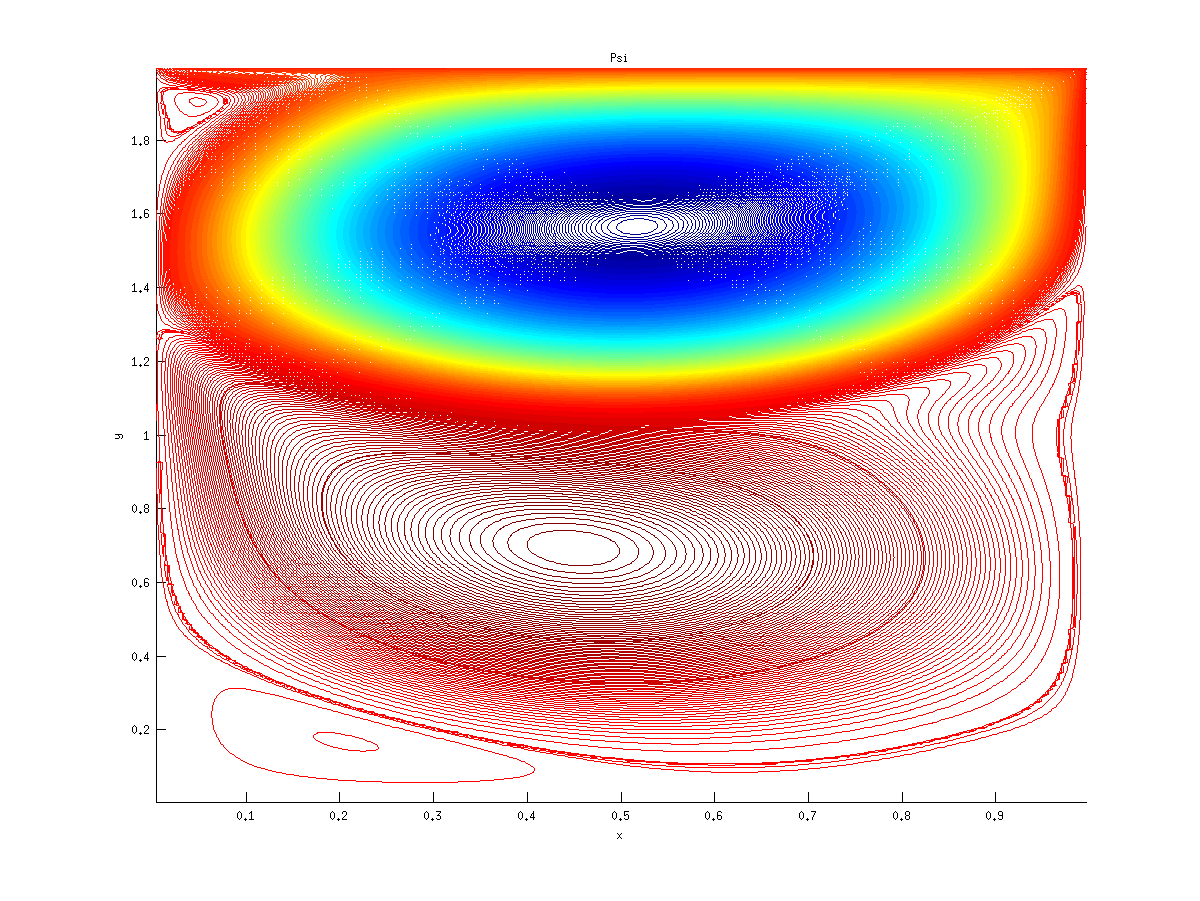}\\
\caption{Solution of NSE (\ref{Navier_Stokes_psi}) on $[0; 1] \times [0; 2]$ - $g \equiv 1$ - $\tau = 1$ -  $255 \times 511$ - $Re = 3200$ - $\Delta t = 0.0005$}
\label{NSrect_Re3200}
\end{center}
\end{figure}
Here the vorticity $\omega$ and the streamfunction $\psi$ at steady state for $Re=3200$. The maximal values found 
for $\psi$ are, for the maximum $0.0196$ located in $(x,y)=(0.4492,06914)$ and the minimum is $-1.215$, located in
$(x,y)=(0.5156,1.15664)$.
\subsection{On the role of $\tau$ and of the extrapolation to the convergence in time to the steady state}
We now denote by $T_c$ the time at which the convergence to the steady state is reached, say $T_c=(k+1)\Delta t$ with
$\displaystyle{ \parallel \Frac{\psi^{(k+1)}-\psi^{(k)} }{\Delta t}\parallel } \le \epsilon$. We put the symbol $***$ when the scheme is unstable and blows up numerically; 'NC' when it is not convergent after A given large time (T=2000) and finally 'Inc. Stab' when it is inconditionally stable.\\

We first consider the RRS scheme with the second order laplacian matrix  as a preconditioner. We show for low Reynolds numbers, that a large value of $\tau$ allows to use large $\Delta t$ and that the slower convergence to the steady state can be corrected by using the Richardson extrapolation.
\begin{table}[h!]
\begin{center}
\begin{tabular}{|c||c|c|c||c|c|c|}
\hline
$\tau$ & $\Delta t$ & $\Delta t_{max}$ & $T_c$ & $\Delta t$ & $\Delta t_{max}$ & $T_c$\\
Extrapolation & no  & no  & no  & yes& yes& yes\\
\hline 
\hline 
$\tau = 1$ & 0.01 & 0.019 & 15.62 & 0.01 & 0.019 & 15.6 \\ 
\hline
              & 0.019 &    & 15.665 & 0.019& & 16.492 \\
\hline 
\hline
$\tau = 10$&  0.02 & 0.14 & 16.86 & 0.02 & 0.11 & 15.8 \\ 
\hline 
  & 0.06 &   & 19.32 & 0.06 &   & 16.8 \\ 
\hline 
  & 0.07 &   & 19.95 & 0.07 &   & 17.29 \\ 
\hline 
  & 0.1 &   & 21.7  & 0.1 &   & 21.7 \\  
\hline 
\hline
$\tau = 30$    & 0.3 & Inc. Stab. & 54.5 & 0.3 &  1.08 & 32.1 \\ 
\hline 
  & 0.4 &   & 79.6 & 0.4 &   & 38 \\ 
\hline 
\hline
\end{tabular}
\caption{RSS (left and RSS with  Extrapolation (right) $Re=100$, $n=63$, $\epsilon=10^{-5}$}
\label{tab1}
\end{center}
\end{table}
\begin{table}[h!]
\begin{center}
\begin{tabular}{|c||c|c|c||c|c|c|}
\hline
$\tau$ & $\Delta t$ & $\Delta t_{max}$ & $T_c$ & $\Delta t$ & $\Delta t_{max}$ & $T_c$\\
Extrapolation & no  & no  & no  & yes& yes& yes\\
\hline 
\hline 
$\tau = 1$ & 0.01 & 0.014 & 35.02& 0.01 & 0.014 & 35.03 \\ 
\hline 
  & 0.014 &   & 35.014  & 0.014 &   & 35.042 \\ 
\hline 
\hline
$\tau = 20$   & 0.2 & 0.28& 67.60 & 0.21 & 0.22 & 57.6 \\ 
\hline 
\hline
$\tau = 30$    & 0.3 &   & 105 & 0.3 &0.35   & 58.2 \\ 
\hline 
                       & 0.35 &   & 117.95 & 0.35 &   & 92.4\\ 
\hline 
\hline 
  $\tau = 50$      & 0.3 & 5.1  & 139.2& 0.3 &  1 & 69 \\ 
\hline         
                          & 0.4 &   & 176.8& 0.4 &   & 83.6 \\ 
\hline         
                          & 0.5 &   & 212.5& 0.5 &   & 99 \\     
\hline         
                          & 0.6 &   & 246.6& 0.6 &   & 113.4 \\                         
\hline 
\hline
\end{tabular}
\caption{RSS (left) and RSS witht   Extrapolation (right) $Re=400$, $n=127$, $\epsilon=10^{-5}$}
\label{tab2}
\end{center}
\end{table}

Of course, since the stabilization allows to take larger time steps, a important gain in CPU time can be obtained when computing a steady state. It can be estimated by considering the number of iteration in time $NT$: $\Frac{T_c}{\Delta t}$ for RSS and  $3\Frac{T_c}{\Delta t}$ for RSS with extrapolation.
For example, taking $Re=400$ and $n=127$, we find $NT=2501$ for $\tau=1$ and $\Delta t=0.014$ and  for $\tau=30$ and $\Delta t=0.6$, we have $NT=410$ (RSS) and $NT=568$ (RSS with extrapolation) ; hence a factor $6.1$ is reached for RSS and one of $4.4$ for extrapolated RSS, see Table \ref{tab2}.
\newpage
We now consider larger Reynolds numbers and take into account the convective part of the equation in the construction of the sparse RSS preconditioner and apply nonlinear RSS scheme (\ref{NLRSS1}) to the vorticity time marching, say
\begin{eqnarray}
\Frac{\omega^{(k+1)}-\omega^{(k)}}{\Delta t} 
+\left(\Frac{1}{Re}A_2+diag(Dy\psi^{(k)}) Dx -diag(Dx\psi^{(k)})Dy\right)(\omega^{(k+1)}-\omega^{(k)})=-F(\psi^{(k)},\omega^{(k)})
\end{eqnarray}
where $A_2$ is the second order laplacian matrix, $diag(Dy\psi^{(k)})$ (resp. $diag(Dx\psi^{(k)})$) is the
diagonal matrix with the discrete (second order accurate) approximation of $\Frac{\partial \psi^{(k)}}{\partial x}$
(resp.  $\Frac{\partial \psi^{(k)}}{\partial y}$) at grid points as entries; $Dx$ (resp. $Dy$) denote the (second order accurate) first derivative matrix in $x$ (resp. in $y$) on the cartesian grid. Finally $-F(\psi^{(k)},\omega^{(k)})$ is the high order compact scheme discretisation of $-\Frac {1}{Re}\Delta \omega
 +\Frac {\partial \phi}{\partial y}
\Frac {\partial \omega}{\partial x} - \Frac {\partial \phi}{\partial
x}
\Frac {\partial \omega}{\partial y}$.
\begin{table}[h!]
\begin{center}
\begin{tabular}{|c||c|c|c||c|c|c||c|c|c|}
\hline
Method & RSS & RSS & RSS & RSS & RSS & RSS & NLRSS & NLRSS & NLRSS\\
$\tau$ & $\Delta t$ & $\Delta t_{max}$ & $T_c$ & $\Delta t$ & $\Delta t_{max}$ & $T_c$&$\Delta t$ & $\Delta t_{max}$ & $T_c$\\
Extrapolation & no  & no  & no  & yes& yes& yes &yes& yes& yes\\
\hline 
\hline 
$\tau = 1$ & 0.005 & 0.005 & 56.21& 0.005 & 0.01 & 56.81 & & &  \\ 
\hline
                  &  0.01       &        &   ***  &     0.01      &  56.79       &   & 0.01 & 0.02 & 56.86 \\ 
                  &  0.02       &        &   ***  &     0.02      & ***       &    & 0.02 &   & 56.96 \\ 
\hline 
\hline  
$\tau = 30$     &0.05 &0.04 & NC  &       0.05&0.08&47.95&  0.05 & 0.7 & 65.05 \\ 
\hline
                     & 0.1   &      &***&  0.1&&***&0.1 &   & 62.5 \\ 
\hline
                     & 0.7   &      &*** &0.7   &&*** &  0.7 &   & 321.3 \\ 
\hline 
\end{tabular}
\caption{RSS (left)   RSS with   Extrapolation (center) and extrapolated NLRSS  (right) $Re=1000$, $n=127$, $\epsilon=10^{-5}$}
\label{tab6}
\end{center}
\end{table}
\begin{table}[h!]
\begin{center}
\begin{tabular}{|c||c|c|c||c|c|c|}
\hline
$\tau$ & $\Delta t$ & $\Delta t_{max}$ & $T_c$ & $\Delta t$ & $\Delta t_{max}$ & $T_c$\\
Extrapolation & no  & no  & no  & no& no& no\\
\hline
$\tau = 10$    & 0.1 &  0.6 & 223.9 &  0.1& 0.006 & ***\\ 
\hline
\end{tabular}
\caption{NLRSS with (left)  and RSS  (right) $Re=3200$, $n=127$, $\epsilon=10^{-5}$}
\label{tab7}
\end{center}
\end{table}

As Shown on Table \ref{tab6} and Table \ref{tab7}, NLRSS outperform RSS for the computation of  Steady States for higher Reynolds numbers, for $Re=1000$ and {\it a fortiori} for $Re=3200$. It allows to use large times steps while RSS is  unstable for such $\Delta t$.
\section{Concluding remarks}
We have studied RRS-like scheme (and their implementations) and pointed out their advantages for the numerical solution of parabolic problems when using high order compact schemes in finite differences for the space disctretization. In particular, the possibility of using fast solvers attached to a standard second order discretization, speeds up the resolution while bringing an enhanced stability.
We also pointed out the role of the
approximation of $\tau B$ to $A$ in the dynamics of the convergence to a steady state: a too strong 
stablization slows down the convergence in time while enhacing the stability of the scheme, the application of Richardson extrapolation allows to recover a dynamics close to the one of the classical schemes.
The robustness of the schemes is illustrated with the solution of 2D NSE equations.
The RSS approach is very versatile and allows adaptations of a large number of techniques of numerical analysis of ODEs. Many developments remain to consider, such as applying factorization updatings on the preconditioners or deriving and applying multilevel general (or Block) RSS schemes for the  solution of other large scale parabolic problems. 

E-mail address: {\tt matthieu.brachet@math.univ-metz.fr}\\
E-mail address: {\tt Jean-Paul.Chehab@u-picardie.fr}\\
\end{document}